\documentclass[11pt,reqno]{amsart}
\usepackage{amssymb,amsmath,amscd,amsthm,amsfonts,wasysym,mathrsfs,amsxtra}
\usepackage{mathtools}
\usepackage{marginnote}
\usepackage{appendix}
\usepackage{hyperref}
\usepackage{tikz}
\usepackage[all,cmtip]{xy}
\usepackage{tikz-cd}

\newcommand{\s}{\mathbb{S}}

\newcommand{\CC}{\mathbb{C}}
\newcommand{\DD}{\mathbb{D}}

\newcommand{\PP}{\mathbb{P}}

\newcommand{\GG}{\mathbb{G}}

\newcommand{\Cc}{\mathcal{C}}
\newcommand{\Ee}{\mathcal{E}}
\newcommand{\Gg}{\mathcal{G}}
\newcommand{\cg}{\mathfrak{g}}

\newcommand{\Dd}{\mathcal{D}}
\newcommand{\Ff}{\mathcal{F}}
\newcommand{\Hh}{\mathcal{H}}
\newcommand{\Kk}{\mathcal{K}}
\newcommand{\Nn}{\mathcal{N}}
\newcommand{\Tt}{\mathcal{T}}
\newcommand{\Pp}{\mathcal{P}}
\newcommand{\Oo}{\mathcal{O}}
\newcommand{\Uu}{\mathcal{U}}

\newcommand{\Qq}{\mathcal{Q}}
\newcommand{\Xx}{\mathcal{X}}
\newcommand{\Yy}{\mathcal{Y}}
\newcommand{\R}{\mathcal{R}}

\newcommand{\V}{\mathcal{V}}

\newcommand{\ZZ}{\mathbb{Z}}

\newcommand{\LL}{\mathscr{L}}
\newcommand{\E}{\sf{E}}
\newcommand{\F}{\sf{F}}
\newcommand{\Hhh}{\sf{H}}
\newcommand{\G}{\sf{G}}
\newcommand{\spi}{\sf{\Psi}}
\newcommand{\SL}{\mathfrak{sl}}
\newcommand{\kk}{\underline{k}}

\newcommand{\bo}{\boldsymbol{1}}
\newcommand{\Vect}{\mathrm{Vect}}
\newcommand{\tdim}{\mathrm{dim}}
\newcommand{\gtdim}{\mathrm{gdim}}

\newcommand{\Hom}{\mathrm{Hom}}
\newcommand{\End}{\mathrm{End}}

\newcommand{\Ext}{\mathrm{Ext}}

\newcommand{\tdet}{\mathrm{det}}

\newcommand{\GLL}{\mathrm{GL}}
\newcommand{\SLL}{\mathrm{SL}}

\newtheorem{theorem}{Theorem}[section]

\newtheorem{lemma}[theorem]{Lemma}

\newtheorem{proposition}[theorem]{Proposition}

\newtheorem{corollary}[theorem]{Corollary}
\theoremstyle{definition}
\newtheorem{definition}[theorem]{Definition}

\theoremstyle{remark}
\newtheorem{remark}[theorem]{Remark}
\newtheorem{conjecture}[theorem]{Conjecture}

\numberwithin{equation}{section}

\title[Exceptional collections, t-structures, and categorical action]{Exceptional collections, t-structures, and categorical action of shifted \MakeLowercase{q}=0 affine algebra, $\SL_{2}$ case}

\begin{document}

\address{National Center for Theoretical Sciences} \email{yhhsu@ncts.ntu.edu.tw}

\author[You-Hung Hsu]{You-Hung Hsu}


\makeatletter
\@namedef{subjclassname@2020}{%
	\textup{2020} Mathematics Subject Classification}
\makeatother


\maketitle

\begin{abstract}
In this article, we show that the categorical action of the shifted $q=0$ affine algebra $\dot{\Uu}_{0,N}(L\SL_{2})$ can be used to construct (or induce) t-structures on the weight categories. The main idea is to interpret the exceptional collection constructed by Kapranov \cite{Kap85}, \cite{Kap88} as convolution of Fourier-Mukai kernels in the categorical action via using the Borel-Weil-Bott theorem.

In particular, when the categories are the bounded derived category of coherent sheaves on Grassmannians. The t-structure we obtain is precisely the exotic t-structure defined by Bezrukavnikov \cite{Be1}, \cite{Be2} of the exceptional collections given by Kapranov. As an application, we calculate the matrix coefficients for generators of the shifted $q=0$ affine algebra on the basis given by Kapranov exceptional collections.
\end{abstract}

\tableofcontents

\section{Introduction}

The derived categories of coherent sheaves on algebraic varieties plays a central role in modern algebraic geometry and related areas. In particular, when the varieties are spaces of importance to geometric representation theory, e.g., Springer fibres, Nakajima quiver varieties, then their (equivariant) derived category of coherent sheaves also provide a good source for categorification of representations of certain Hecke algebras or quantum groups.

There are some important devices from homological algebras like exceptional objects and t-structures that would be frequently used to study them.  Also, these devices have applications to geometric representation theory for the varieties that we mention above, for example, see \cite{Be3}.

\subsection{Motivation}

In the paper \cite{Hsu}, the author introduce the shifted $q=0$ affine algebra $\dot{\Uu}_{0,N}(L\SL_{n})$, and he give a definition of its categorical action. Then he showed that there is a categorical action of $\dot{\Uu}_{0,N}(L\SL_{n})$ on the bounded derived category of coherent sheaves on partial flag varieties $\bigoplus_{\kk} \Dd^b(Fl_{\kk}(\CC^N))$. In this paper we focus on the $n=2$ case; i.e., we have the categorical action of $\dot{\Uu}_{0,N}(L\SL_{2})$ on the bounded derived category of coherent sheaves on Grassmannians $\bigoplus_{k} \Dd^b(\GG(k,N))$.

We try to relate the notion of categorical action defined in \textit{loc. cit.} to the exceptional objects and t-structures on those weight categories.  More precisely, since we constructed a categorical action of $\dot{\Uu}_{0,N}(L\SL_{2})$ on $\bigoplus_{k} \Dd^b(\GG(k,N))$, when it pass to the Grothendieck group $\bigoplus_{k} K(\GG(k,N))$ we obtain a representation of $\dot{\Uu}_{0,N}(L\SL_{2})$. Usually in representation theory, it would be good to find a basis for a representation such that generators of the algebra have a particularly simple action on it, e.g., crystal bases or canonical bases.

So the motivation of this paper comes from finding a good or natural basis for the Grothendieck group $K(\GG(k,N))$ and to study how the generators of $\dot{\Uu}_{0,N}(L\SL_{2})$ acting on them. Finally, throughout this article, all functors between derived categories are assumed to be derived functors, e.g., we will write $f^*$ instead of $Rf^*$ and so on.

\subsection{Basis from exceptional collections}

The Grassmannians and partial flag varieties are homogeneous spaces; i.e., they are certain quotient $G/P$ of a complex semisimple algebraic group $G$ by a parabolic subgroup $P$. The structure of their bounded derived category of coherent sheaves, denoted by $\Dd^b(G/P)$, has been studied over many years. In particular, by Beilinson, Kapranov \cite{B}, \cite{Kap85}, \cite{Kap88} for type A, and by Kuznetsov, Polishchuk \cite{KP} for other general types. Their studies lead to the construction of exceptional collections on $\Dd^b(G/P)$. We also refer to \cite{Ku} and \cite {PS} for some cases studies, and \cite{BLVdB} and \cite{E} for a characteristic free point of view.

Since the geometry we use to construct the categorical action is the partial flag varieties of type A,  the exceptional collections were constructed by Beilinson and Kapranov, see Corollary \ref{Corollary 1}, \ref{Corollary 2}. Moreover, those exceptional collections are full; i.e., the smallest triangulated subcategory containing the exceptional collection is the whole derived category. Thus when it pass to the Grothendieck group, we get a basis for the representation.

Let $P(a,b)$ the set of Young diagrams $\lambda$ such that $\lambda_{1} \leq a$ and $\lambda_{b+1} =0$. From Corollary \ref{Corollary 2}, the exceptional collection for $\Dd^b(\GG(k,N))$ is given by $\{\s_{\lambda}\V \ | \ \lambda \in P(N-k,k) \}$, where $\s_{\lambda}$ is the Schur functor associate to the Young diagram $\lambda$. Here $\V$ is the tautological rank $k$ bundle on $\GG(k,N)$. Passing to the Grothendieck group, we obtain the basis $\{[\s_{\lambda}\V]\ | \ \lambda \in P(N-k,k) \}$ for $K(\GG(k,N))$.

So we try to understand how the main generators $e_{r}1_{(k,N-k)}$ and $f_{s}1_{(k,N-k)}$ of $\dot{\Uu}_{0,N}(L\SL_{2})$ acting on $\{[\s_{\lambda}\V]\}$ for all $r,s \in \ZZ$. Recall that the main generators of $\dot{\Uu}_{0,N}(L\SL_{2})$ acting on $\bigoplus_{k} \Dd^b(\GG(k,N))$ is defined by using the following correspondence diagram 
\begin{equation} \label{diagram1}   
	\xymatrix{ 
		&&Fl(k-1,k)=\{0 \overset{k-1}{\subset} V' \overset{1}{\subset} V \overset{N-k}{\subset} \CC^N \} 
		\ar[ld]_{p_1} \ar[rd]^{p_2}   \\
		& \GG(k,N)  && \GG(k-1,N)
	}
\end{equation} where $Fl(k-1,k)$ is the $3$-step partial flag variety and the numbers above the inclusion indicate the increasing of dimensions. Here $p_{1}$ and $p_{2}$ are natural projections. Let $\V,\V'$  be the tautological bundles on $Fl(k-1,k)$ of rank $k$, $k-1$, respectively. Then there is a natural line bundle $\V/\V'$ on $Fl(k-1,k)$.  The generators $e_{r}1_{(k,N-k)}$ acting on $\bigoplus_{k} \Dd^b(\GG(k,N))$ by lifting to the following functor 
\begin{equation*}
{\E}_{r}\bo_{(k,N-k)}:=p_{2*}(p_{1}^{*} \otimes (\V/\V')^{r}):\Dd^b(\GG(k,N)) \rightarrow \Dd^b(\GG(k-1,N))
\end{equation*} and similarly for the lift of  $f_{s}1_{(k,N-k)}$ to ${\F}_{s}\bo_{(k,N-k)}$.

After decategorifying, it is not easy to calculate the action of generators on this basis directly by definition. For example, when $r=0$, we have to calculate ${e}1_{(k,N-k)}([\s_{\lambda}\V])=p_{2*}(p_{1}^{*}([\s_{\lambda}\V]))$ for all $\lambda$, which may take an amount of efforts.

To solve this, we lift it to the categorical level, which means we study the categorical action $\dot{\Uu}_{0,N}(L\SL_{2})$ on the exceptional collection $\{\s_{\lambda}\V\}$. The key observation is to interpret $\s_{\lambda}\V$ as convolution of Fourier-Mukai kernels by using the relative version of Borel-Weil-Bott theorem, see Theorem \ref{Theorem 5} and Corollary \ref{Corollary 6} in Section \ref{subsection 5.1} for details. 

Since the functor ${\E}\bo_{(k,N-k)}$ is a Fourier-Mukai transformation with kernel given by $\Ee\bo_{(k,N-k)}=\iota_{*}\Oo_{Fl(k-1,k)}$, where $\iota:Fl(k-1,k) \rightarrow \GG(k,N) \times \GG(k-1,N)$. The calculation of ${\E}\bo_{(k,N-k)}(\s_{\lambda}\V)$ is given by convolution of Fourier-Mukai kernels, which can be calculated by using  the relations in the definition of categorical action. As an application, when we pass to the Grothendieck group, we obtain the result for ${e}1_{(k,N-k)}([\s_{\lambda}\V])$. Similarly for the functor ${\F}1_{(k,N-k)}$ and ${f}1_{(k,N-k)}([\s_{\lambda}\V])$, see Proposition \ref{Proposition 3} and Corollary \ref{Corollary 7} in Section \ref{subsection 6.1} for details.

However, for ${e}_{r}1_{(k,N-k)}$ and ${f}_{s}1_{(k,N-k)}$ with $r, s \neq 0$, the above method fails. Since lifting to the categorical levels, the calculation of ${\E}_{r}\bo_{(k,N-k)}(\s_{\lambda}\V)$ and ${\F}_{s}\bo_{(k,N-k)}(\s_{\lambda}\V)$ involves relations that are not defined in the definition of categorical action. We solve this by the help of dual exceptional collections, see Definition \ref{definition 8} for the definition. By definition, after passing to the Grothendieck group, the dual exceptional collection can be thought as a dual basis for the exceptional collection with respect to the bilinear Euler form $\sum_{i \in \ZZ} (-1)^i \dim_{\CC} \Hom_{\Dd} ( - , - [i] )$.

The dual exceptional collection for $\{\s_{\lambda}\V\}$ is given by $\{\s_{\mu}(\CC^N/\V)[-|\mu|]\ | \ \mu \in P(k,N-k) \}$, and $|\mu|$ is the sum of all boxes in $\mu$.  So every element $[\Xx] \in K(\GG(k,N))$ can be written as 
\begin{equation} \label{eq 1} 
	[\Xx]=\sum_{\lambda \in P(N-k,k) } (-1)^{|\lambda|} \gtdim R\Hom(\s_{\lambda^{*}}(\CC^N/\V), \Xx) [\s_{\lambda}\V]
\end{equation} where $\lambda^*$ is the conjugate partition of $\lambda$ and $\gtdim$ is the notation for graded dimension, i.e., 
\begin{equation*}
	\gtdim R\Hom(\Xx, \Yy):=\sum_{i \in \ZZ} (-1)^{i} \tdim \Ext^{i}(\Xx,\Yy)
\end{equation*} for all $\Xx, \Yy \in \Dd^b(\GG(k,N))$.

Now the key observation is ${e}_{r}1_{(k,N-k)}$ can be written as $e1_{(k,N-k)}$ conjugated by powers of determinant line bundles. More precisely, we let $g:K(\GG(k,N)) \rightarrow K(\GG(k,N))$ be the linear map given by multiplication with $[\det(\V)]$. Then $e_{r}1_{(k,N-k)}=g^{-r}eg^{r}1_{(k,N-k)}$ for all $r \in \ZZ$.  So to calculate ${e}_{r}1_{(k,N-k)}([\s_{\lambda}])$ with $r \neq 0$, we first apply $g^{r}1_{(k,N-k)}$ to $[\s_{\lambda}\V]$ and then using (\ref{eq 1}) to express $[\det(\V)]^{r}[\s_{\lambda}\V]$ in terms of linear combination of the basis $\{[\s_{\lambda}\V]\ | \ \lambda \in P(N-k,k) \}$. Next, we apply $e1_{(k,N-k)}$ to the linear combination and using the result we already know for ${e}1_{(k,N-k)}([\s_{\lambda}\V])$ to simplify it. Then applying $g^{-r}1_{(k-1,N-k+1)}$ to the rest terms and using (\ref{eq 1}) again to express those terms as linear combination of the basis $\{[\s_{\lambda'}\V']\ | \ \lambda' \in P(N-k+1,k-1) \}$. Finally, we have determine the coefficients for each $[\s_{\lambda'}\V']$ to be non-zero by using the Borel-Weil-Bott Theorem. Similarly for ${f}_{s}1_{(k,N-k)}$ with $s \neq 0$. We refer the readers to Theorem \ref{Theorem 10} and \ref{Theorem 11} for details. It would be interesting to see the geometric and representation-theoretic meanings behind these matrix coefficients.

\subsection{Relation to t-structures}

When we study the categorical action on the exceptional collections given by Beilinson and Kapranov,  we also found it relates to theory of t-structures.  Roughly speaking, there are two descriptions of a t-structure on the bounded derived category of coherent sheaves on Grassmannians. One is defined by using the exceptional collection given by Kapranov, and the other is defined by using the categorical action of $\dot{\Uu}_{0,N}(L\SL_{2})$.

Let us explain a bit about the motivation. Denote $\cg$ be the Lie algebra of a complex simple algebraic group $G$, and  $\Nn \subset \cg$ be the subvariety of nilpotent elements. In \cite{Be1}, R. Bezrukavnikov constructed a new t-structure, which he called it the exotic t-structure, on the derived category of $G$-equivariant coherent sheaves on $\Nn$. This t-structure has two descriptions, one is defined by the perverse t-structure on the equivariant coherent sheaves, the other is defined by a quasi-exceptional set, where the conditions are a bit weaker than the definition of a exceptional collection. Then he proved that the two descriptions coincide, and thus give a elementary way to prove the Lusztig-Vogan bijection.
 
Since we have the exceptional collection and its dual exceptional collection given by Kapranov on $\Dd^b(\GG(k,N))$. By R. Bezrukavnikov's work \cite{Be1} (see Theorem \ref{Theorem 6}), we obtain an exotic t-structure on these categories $\Dd^b(\GG(k,N))$, denoted by $(\Dd^{\geq 0}_{ex}(k,N-k), \Dd^{\leq 0}_{ex}(k,N-k))$, see Corollary \ref{Corollary 4}.

On the other hand, peoples also use categorical actions of affine braid groups and quantum affine algebras to defined t-structures. For example, Bezrukavnikov and Mirkovic \cite{BM} defined exotic t-structures on derived category of coherent sheaves on varieties related to the Springer resolution $\pi:\widetilde{\Nn} \rightarrow \Nn$ by using a certain affine braid group actions that developed in \cite{BR} on these categories.

Similarly, Cautis-Koppensteiner \cite{CKo} defined exotic t-structures on derived category of coherent sheaves on convolution varieties related to Beilinson-Drinfeld and affine Grassmannians by using categorical action of quantum affine algebras. Their work also recover the  t-structures in \cite{BM}  and \cite{BMR}.

Motivated by this, it is natural to ask whether the exotic t-structure $(\Dd^{\geq 0}_{ex}(k,N-k), \Dd^{\leq 0}_{ex}(k,N-k))$ can be described by the categorical action of $\dot{\Uu}_{0,N}(L\SL_{2})$ that defined by the author. 

We achieve this by first  showing that given an abstract  categorical action of $\dot{\Uu}_{0,N}(L\SL_{2})$ on $\bigoplus_{k}\Kk(k,N-k)$ and a fixed t-structure on the highest weight category $\Kk(0,N)$. Then with some extra assumptions we can use the action to construct (or induce) t-structure on the other weight categories $\Kk(k,N-k)$, see Theorem \ref{Theorem 7} for details. 

The main tool we use is a result by Polishchuk \cite{Po}, see Theorem \ref{Theorem 8}. Basically, we have to construct functors $\varphi_{k}:\Kk(0,N) \rightarrow \Kk(k,N-k)$ such that $(\varphi_{k})^{R} \circ \varphi_{k}$ is t-exact for all $k$, then we translate the t-structure from $\Kk(0,N)$ to $\Kk(k,N-k)$ via $\varphi_{k}$. The definition of the functors $\varphi_{k}$ can be thought as an abstracting of the tilting bundle $\bigoplus_{\lambda} \s_{\lambda}\V$ of the Kapranov exceptional collection.  Moreover, we can show that the functors $({\F}_{N-k-1}\bo_{(k,N-k)})^{R}$ and $({\E}\bo_{(k,N-k)})^{R}({\spi}^{-})[k-2]$ are t-exact on these t-structures. 

As a corollary, by modifying the functor $\varphi_{k}$ we can obtain another t-structure on $\Kk(k,N-k)$ such that $({\F}_{N-k-1}\bo_{(k,N-k)})^{L}$ and $({\E}\bo_{(k,N-k)})^{L}({\spi}^{-})[k-2]$ are t-exact on these t-structures, see Corollary \ref{Corollary 8}. 

Then we apply to the case where the weight categories are $\Dd^b(\GG(k,N))$, the highest weight category $\Dd^b(\GG(0,N))$ is the category of graded vector spaces. We equip the highest weight category with the standard t-structure, then we induce a t-structure on $\Dd^b(\GG(k,N))$ by Corollary \ref{Corollary 8}, denoted by $(\Dd^{\geq 0}_{act}(k,N-k), \Dd^{\leq 0}_{act}(k,N-k))$

Now we have two t-structures on $\Dd^b(\GG(k,N))$, we prove that the two t-structures agree, i.e. $\Dd^{\geq 0}_{ex}(k,N-k)= \Dd^{\geq 0}_{act}(k,N-k)$ and $\Dd^{\leq 0}_{ex}(k,N-k)= \Dd^{\leq 0}_{act}(k,N-k)$, see Theorem \ref{Theorem 9}.

\subsection{Some further remarks}

We mention some remarks about topics that we do not address in this article but would like to study them in the future.

First, since all the results in this paper are for the $\SL_{2}$ case, we would like to generalize them to the $\SL_{n}$ case in the future.

Second, in \cite{Be1}, \cite{Be2}, Bezrukavnikov also study the heart of the t-structure of a quasi-exceptional set. In particular, he give a description of the irreducible objects in this heart. However, in this paper, we do not study the heart of the t-structure of the Kapranov exceptional collection. So we would like to understand the irreducible objects in this t-structure, especially its relation to the categorical action of shifted $q=0$ affine algebra.

Also, since the Kapranov exceptional collection is strong and full, the direct sum of the exceptional objects in this exceptional collection, i.e., $\Tt:=\bigoplus_{\lambda}  \s_{\lambda}\V$, is a tilting bundle on $\GG(k,N)$. We have the following equivalence
\begin{equation} \label{eq 30}
R\Hom(\Tt, ):\Dd^b(\GG(k,N)) \xrightarrow{\cong} \Dd^b(A^{op}-mod)
\end{equation} where $A:=\End(\Tt)$.

We expect that under this derived equivalence, the heart of this t-structure is equivalent to the abelian category of (right) modules over the endomorphism algebra of the tilting bundle. More precisely, we have the following conjecture.

\begin{conjecture}
Under the above equivalence (\ref{eq 30}), the restriction of the functor to the heart of the exotic t-structure gives an equivalence between abelian categories.
\begin{equation*}
R\Hom(\Tt, ):\Dd^{\geq 0}_{ex}(k,N-k) \cap \Dd^{\leq 0}_{ex}(k,N-k) \xrightarrow{\cong} A^{op}-mod
\end{equation*}
\end{conjecture}

Next, recall that in \cite{BM}, they defined the exotic t-structures by using the affine braid group action and define the notion of braid positive for those t-structures. According to \cite{CK}, categorical $\cg$ actions can be used to construct braid group actions, and from Theorem 4.1 in \cite{CKo}, they show that the t-structure constructed from the categorical action of quantum affine algebra is braid positive. 

For categorical action of shifted $q=0$ affine algebra, the author do not know whether such an action can be used to construct braid group actions or not. However, we still expect that such an action can be used to construct an action that is similar to braid group, e.g., braid monoid, and to understand more about the t-structure of the Kapranov exceptional collection.

Finally, in \cite{KP}, Kuznetsov-Polishchuk constructed exceptional collections for general type Grassmannians and flag varieties $G/P$, they conjecture that such exceptional collections are full and strong. Since in this paper we interpret objects in the Kapranov exceptional collection as convolution of Fourier-Mukai kernels, we also expect that the shifted $q=0$ affine algebras for general type (not defined yet) can be used to partially address the conjectures by Kuznetsov-Polishchuk.

\subsection{Organization}
This article is organized as follows.

In Section \ref{Section 2}, we recall the definitions of exceptional collections and Fourier-Mukai transformations. After this we recall the full exceptional collections constructed by Beilinson (for projective spaces) \cite{B} and Kapranov (for Grassmannians) \cite{Kap85}, see Corollary \ref{Corollary 1}, \ref{Corollary 2}.

In Section \ref{Section 3}, we recall the definition of shifted $q=0$ affine algebras $\dot{\Uu}_{0,N}(L\SL_{2})$ and its categorical action. Then we mention the result where there is a categorical action of $\dot{\Uu}_{0,N}(L\SL_{2})$ on $\bigoplus_{k}\Dd^b(\GG(k,N))$ \cite{Hsu}, see Theorem \ref{Theorem 3}.

In Section \ref{Section 4}, we recall the definition of t-structures and some technical tools we need about t-structures. Then we mention a result by Bezrukavnikov \cite{Be1} (Theorem \ref{Theorem 6}) which says we can construct a t-structure from an exceptional collections and its dual. Then we apply this result to the Kapranov exceptional collections and obtain a t-structure on $\Dd^b(\GG(k,N))$, see Corollary \ref{Corollary 4}.

In Section  \ref{Section 5}, first we mention the Borel-Weil-Bott theorem (Theorem \ref{Theorem 4}) and gives a relative version of it, see Theorem \ref{Theorem 5}. Then we relate the Kapranov exceptional collections and the categorical action of $\dot{\Uu}_{0,N}(L\SL_{2})$ on $\bigoplus_{k}\Dd^b(\GG(k,N))$, see Corollary \ref{Corollary 6}. Next we prove the main theorem of this article, i.e.; Theorem \ref{Theorem 7}, which says we can use categorical action of $\dot{\Uu}_{0,N}(L\SL_{2})$ to induce t-structures on weighted categories. Finally we apply this result to $\Dd^b(\GG(k,N))$ to indue a t-structure on it, and we show that the induced t-structure is precisely the same as the t-structure given by the Kapranov exceptional collections in Corollary \ref{Corollary 4}, see Theorem \ref{Theorem 9}.

In Section \ref{Section 6}, we pass to the Grothendieck group and the Kapranov exceptional collection gives a basis for $K(\GG(k,N))$. We calculate the matrix coefficients for the main generators $e_{r}1_{(k,N-k)}$, $f_{s}1_{(k,N-k)}$ of $\dot{\Uu}_{0,N}(L\SL_{2})$ acting on this basis. We calculate the case $r=s=0$ first. The main idea is to calculate the action on categorical level first via using Corollary \ref{Corollary 6} to write the action in terms of convolution of kernels. Then passing to the Grothendieck group to obtain the result, see Corollary \ref{Corollary 7}. Finally we calculate the case $r,s \neq 0$ by using Corollary \ref{Corollary 7} and the help of dual exceptional collection and Borel-Weil-Bott theorem, see Theorem \ref{Theorem 10}, \ref{Theorem 11} for details.

\subsection{Acknowledgements}
The author would like to thank Prof. Wu-yen Chuang and Cheng-Chiang Tsai for many helpful discussions. He also want to thank Prof. Sabin Cautis and Clemens Koppensteiner for some valuable correspondences.

\section{Generators for the derived categories of coherent sheaves on Grassmannians } \label{Section 2}

In this section, we recall the works of A. Beilinson \cite{B} and M. Kapranov \cite{Kap85}, \cite{Kap88} about the structures of bounded derived categories of coherent sheaves on projective space, Grassmannians. More precisely, we give the exceptional collections for their bounded derived categories of coherent sheaves.

\subsection{Schur functors and tensor products}

First, we recall briefly a basic tool from representation theory which would be used in later sections. 

The materials for Schur functors follows from Chapter 4 in \cite{FH}. Let $\lambda=(\lambda_{1},...,\lambda_{n})$ be a non-increasing sequence of positive integers. We can represent $\lambda$ as a Young diagram with $n$ rows, aligned on the left, such that the $i$th row
has exactly $\lambda_{i}$ cells. The size of $\lambda$, denoted by $|\lambda|$, is the number $|\lambda|=\sum_{i=1}^{n} \lambda_{i}$. The transpose diagram $\lambda^*$ is obtained by exchanging rows and columns of $\lambda$.

\begin{definition}
Let $n \geq 1$ be a positive integer and $\lambda=(\lambda_{1},...,\lambda_{n})$ be a sequence of non-increasing positive integers. The Schur functor $\s_{\lambda}$ associated to $\lambda$ is defined as a functor 
\[
\s_{\lambda}:\Vect_{\CC} \rightarrow \Vect_{\CC}
\] 
such that for any vector space $V$, $\s_{\lambda}V$ coincides with the image of the Young symmetrizer $c_{\lambda}$ in the space of tensors of $V$ of rank $n$: i.e., $\s_{\lambda}V=\text{Im} (c_{\lambda}|_{V^{\otimes n}})$.
\end{definition}

An useful application of the Schur functors comes out in the description of the spaces of
exterior power of a tensor product, which is the following lemma.

\begin{lemma} [\cite{Kap85} Lemma 0.5] \label{lemma 1}
Let $V$ and $W$ be two vector spaces and let $p$ be a non-negative integer. Then we have a natural isomorphism of $\GLL(V) \times \GLL(W)$-modules
\[
\bigwedge^{p}(V \otimes W) \cong \bigoplus_{|\lambda|=p} \s_{\lambda}V \otimes \s_{\lambda^*} W 
\] where $\lambda$ runs all Young diagram of size $p$ and each summand occurs with multiplicity one.
\end{lemma}

\subsection{Exceptional collections and Fourier-Mukai transformations}

In this section, we briefly recall the definition of exceptional objects/collections and Fourier-Mukai transformations. We follow the book \cite{Huy} for the definitions and details. 

Let $\Dd$ be a $\CC$-linear triangulated category. 

\begin{definition} 
An object $E \in \text{Ob}(\Dd)$ is called exceptional if 
\[
\Hom_{\Dd}(E,E[l])=\begin{cases}
\CC & \ \text{if} \ l=0 \\
0 & \ \text{if} \ l \neq 0.
\end{cases}
\]
\end{definition}

Then we define the notion of exceptional collections.

\begin{definition}
An ordered collection (sequence) $\{E_{1},...,E_{n}\}$, where $E_{i} \in \text{Ob}(\Dd)$ for all $1 \leq i \leq n$,  is called an exceptional collection (sequence) if each $E_{i}$ is exceptional and 
\[
\Hom_{\Dd}(E_{i},E_{j}[l])=\begin{cases}
\CC & \ \text{if} \ l=0, \ i=j \\
0 & \ \text{if} \ i>j \ \text{or if} \ l \neq 0, \ i=j.
\end{cases}
\] 
The collection (sequence) is called strong exceptional if in addition
\[
\Hom_{\Dd}(E_{i},E_{j}[l])=0 \ \text{for all} \ i, \ j \ \text{and} \ l \neq 0.
\]
\end{definition}

\begin{remark}
An exceptional collection (sequence) $\{E_{1},...,E_{n}\}$ is called full if the smallest triangulated subcategory of $\Dd$ containing all the objects $E_{1},...,E_{n}$ is $\Dd$ itself.
\end{remark}

Next, we introduce the language of Fourier-Mukai transformations. For $X$  a smooth complex projective variety, we will work with the bounded derived category of coherent sheaves on $X$, which we always denote it by $\Dd^b(X)$. Throughout this article, all pullback, pushforward, Hom, and tensor product functors of sheaves will be derived functors. 

\begin{definition} 
	Let $X$, $Y$ be two smooth projective varieties. A Fourier-Mukai kernel is any object $\Pp$ of the
	derived category of coherent sheaves on $X \times Y$ . Given $\Pp \in \Dd^b(X \times Y)$, we may define
	the associated Fourier-Mukai transform, which is the functor
	\[
	\Phi_{\Pp}:\Dd^b(X) \rightarrow \Dd^b(Y) 
	\]
	\[
	\ \ \ \ \ \ \ \ \Ff \mapsto \pi_{2*}(\pi_{1}^*(\Ff) \otimes \Pp)
	\] where $\pi_{1}$, $\pi_{2}$ are natural projections.
\end{definition}

We call $\Phi_{\Pp}$ the Fourier-Mukai transform with (Fourier-Mukai) kernel $\Pp$. For convenience, we would just write FM for Fourier-Mukai.

The first property of FM transforms is that they have left and right adjoints that are themselves FM transforms.

\begin{proposition} [\cite{Huy} Proposition 5.9] \label{Proposition 1}
	For $\Phi_{\Pp}:\Dd^b(X) \rightarrow \Dd^b(Y)$ is the FM transform with kernel $\Pp$, define 
	\[
	\Pp_{L}=\Pp^{\vee} \otimes \pi^*_{2}\omega_{Y}[\tdim  Y], \ \Pp_{R}=\Pp^{\vee} \otimes \pi^*_{1}\omega_{X}[\tdim X].
	\] Then
	\[
	\Phi_{\Pp_{L}}:\Dd^b(Y) \rightarrow \Dd^b(X), \  \Phi_{\Pp_{R}}:\Dd^b(Y) \rightarrow \Dd^b(X)
	\] are the left and right adjoints of $\Phi_{\Pp}$, respectively.
\end{proposition}

The second property is the composition of FM transforms is also a FM transform.

\begin{proposition} [\cite{Huy} Proposition 5.10] \label{Proposition 2}
	Let $X, Y, Z$ be smooth projective varieties over $\CC$. Consider objects $\Pp \in \Dd^b(X \times Y)$ and $\Qq \in \Dd^b(Y \times Z)$. They define FM transforms $\Phi_{\Pp}:\Dd^b(X) \rightarrow \Dd^b(Y)$, $\Phi_{\Qq}:\Dd^b(Y) \rightarrow \Dd^b(Z)$.  We would use $\ast$ to denote the operation for convolution, i.e.
	\[
	\Qq \ast \Pp:=\pi_{13*}(\pi_{12}^*(\Pp)\otimes \pi_{23}^{*}(\Qq))
	\]
	Then for $\R=\Qq \ast \Pp \in \Dd^b(X \times Z)$, we have $\Phi_{\Qq} \circ \Phi_{\Pp} \cong \Phi_{\R}$. 
\end{proposition}

\begin{remark} \label{remark 1}
	Moreover by \cite{Huy} remark 5.11, we have $(\Qq \ast \Pp)_{L} \cong (\Pp)_{L} \ast (\Qq)_{L}$ and $(\Qq \ast \Pp)_{R} \cong (\Pp)_{R} \ast (\Qq)_{R}$.
\end{remark}

\subsection{The complex projective space $\PP^N$}

Let $\Dd^b(\PP^N)$ be the bounded derived categories of coherent sheaves on the complex projective space $\PP^N$. In this subsection, we introduce a strong full exceptional collections for $\Dd^b(\PP^N)$, which is due to A. Beilinson \cite{B}.

Consider $\Oo_{\Delta}=\Delta_{*}\Oo_{\PP^N} \in \Dd^b(\PP^N \times \PP^N)$, which is the structure sheaf of the diagonal in $\PP^N \times \PP^N$, where $\Delta:\PP^N \rightarrow \PP^N \times \PP^N$ is the diagonal embedding. A. Beilinson constructed a locally free resolution of $\Oo_{\Delta}$ that can be described by two distinguished families of sheaves on $\PP^N$.

Let us fix some notations. Denote $\Oo_{\PP^N}(-1)$ to be the tautological line bundle on $\PP^N$, and define $\Oo_{\PP^N}(1):=\Oo_{\PP^N}(-1)^{\vee}$, $\Oo_{\PP^N}(i):=\Oo_{\PP^N}(1)^{\otimes i}$ for $i \in \ZZ$. Similarly, denote $\Omega_{\PP^N}$ to be the cotangent bundle (sheaf of differentials) of $\PP^N$, and define $\Omega^{i}_{\PP^N}:=\bigwedge^{i}\Omega_{\PP^N}$ for $i \geq 1$.

Finally, for a smooth complex projective variety $X$, if $\Ff, \ \Gg \in \Dd^b(X)$, then we denote $\Ff \boxtimes \Gg = \pi_{1}^{*}(\Ff) \otimes \pi_{2}^*(\Gg)$, where $\pi_{1}, \ \pi_{2}$ are natural projections from $X \times X$ to $X$. 

Then the resolution of $\Delta_{*}\Oo_{\PP^N} $ can be described as follows.

\begin{theorem} \cite{B} \label{Theorem 1}
With the above notations, $\Oo_{\Delta}$ admits a locally free resolutions in $ \Dd^b(\PP^N \times \PP^N)$.
\begin{equation}
0 \rightarrow \bigwedge^{N} (\Oo_{\PP^N}(-1) \boxtimes \Omega_{\PP^N}(1))  \rightarrow  ... \rightarrow \Oo_{\PP^N}(-1) \boxtimes \Omega_{\PP^N}(1) \rightarrow \Oo_{\PP^N \times \PP^N} \rightarrow \Oo_{\Delta} \rightarrow 0.
\end{equation}
\end{theorem}

\begin{remark} 
For a  proof, we refer to \cite{Cal} Theorem III.1.1 for details.
Note that by Lemma \ref{lemma 1}, we have the isomorphism $\bigwedge^{i} (\Oo_{\PP^N}(-1) \boxtimes \Omega_{\PP^N}(1)) \cong \Oo_{\PP^N}(-i) \boxtimes \Omega^{i}_{\PP^N}(i)$ for all $i \geq 1$. 
\end{remark}

Moreover, we have the following corollary.

\begin{corollary}  [\cite{B}] \label{Corollary 1}
The following collection of sheaves 
\begin{align*}
&\{\Oo_{\PP^N}(-N),...,\Oo_{\PP^N}(-1),\Oo_{\PP^N}\} 
\end{align*} 
is a strong full exceptional collection in $\Dd^b(\PP^N)$.
\end{corollary}

\begin{remark}
More generally, we can have $\{\Oo_{\PP^N}(a),...,\Oo_{\PP^N}(a+N-1),\Oo_{\PP^N}(a+N)\}$ is a strong full exceptional collection for any $a \in \ZZ$, see \cite{Huy} Corollary 8.29.
\end{remark}

\subsection{The Grassmannians $\GG(k,N)$}

The above result Theorem \ref{Theorem 1}, Corollary \ref{Corollary 1} by Beilinson were generalized by Kapranov \cite{Kap85}, \cite{Kap88} to the Grassmannians and partial flag varieties.

First, we define the Grassmannians 
\[
\GG(k,N)=\{0 \subset V \subset \CC^N \ | \ \tdim_{\CC} V=k\}
\] of $k$-dimensional subspaces in $\CC^N$.

Let $\Dd^b(\GG(k,N))$  be the bounded derived categories of coherent sheaves on the Grassmannian $\GG(k,N)$. Similarly as the case for $\PP^N$, we introduce a strong full exceptional collection for $\Dd^b(\GG(k,N))$.

Denote $\V$ to be the tautological rank $k$ bundle on $\GG(k,N)$ and $\CC^N/\V$ to be the tautological rank $N-k$ quotient bundle. Similarly as Theorem \ref{Theorem 1}, the projective space can be generalized to the Grassmannians case. This means that there is a locally free resolution of the structure sheaf $\Oo_{\Delta}=\Delta_{*}\Oo_{\GG(k,N)}$, where $\Delta:\GG(k,N) \rightarrow \GG(k,N) \times \GG(k,N)$ is the diagonal embedding. The following is due to M. Kapranov.

\begin{theorem} [\cite{Kap85}]\label{Theorem 2}
	With the above notations, $\Oo_{\Delta}$ admits a locally free resolutions in $ \Dd^b(\GG(k,N) \times \GG(k,N))$.
	\begin{equation}
	0 \rightarrow \bigwedge^{k(N-k)} (\V \boxtimes (\CC^N/\V)^{\vee})  \rightarrow ... \rightarrow  \V \boxtimes (\CC^N/\V)^{\vee} \rightarrow \Oo_{\GG(k,N) \times \GG(k,N)} \rightarrow \Oo_{\Delta} \rightarrow 0.
	\end{equation}
\end{theorem}

\begin{remark}
Note that by Lemma \ref{lemma 1}, we have the isomorphism 
\[
\bigwedge^{i} (\V \boxtimes (\CC^N/\V)^{\vee})  \cong \bigoplus_{|\lambda|=i} \s_{\lambda}\V \boxtimes \s_{\lambda^*}(\CC^N/\V)^{\vee}
\] for all $i \geq 1$. 
\end{remark}

For non-negative integers $a, \ b \geq 0$, we denote by $P(a,b)$ the set of Young diagrams $\lambda$ such that $\lambda_{1} \leq a$ and $\lambda_{b+1} =0$.
Then, we obtain the following  strong full exceptional collection.

\begin{corollary} [\cite{Kap85}]  \label{Corollary 2} 
	The following collection of sheaves 
	\begin{align*}
	&\R_{(k,N-k)}=\{\s_{\lambda}\V\ | \ \lambda \in P(N-k,k) \} 
	\end{align*}
	is a strong full exceptional collection in $\Dd^b(\GG(k,N))$. 
\end{corollary}

\begin{remark}
When $k=1$, it is easy to see that the exceptional collections $\R_{(1,N-1)}$ agree with the exceptional collections for $\PP^{N-1}$ from Corollary \ref{Corollary 1}.
\end{remark}

\section{Shifted $q=0$ affine algebras and its categorical action} \label{Section 3}

In this section, we recall the definitions of shifted $q=0$ affine algebras $\dot{\Uu}_{0.N}(L\SL_{2})$ and their categorical actions. We also recall the result for its categorical action on the bounded derived categories of coherent sheaves on Grassmannians \cite{Hsu}.

\subsection{Shifted $q=0$ affine algebras}

\begin{definition} \label{definition 1}
	Denote by $\dot{\Uu}_{0,N}(L\SL_2)$ the associative $\CC$-algebra generated by
	\begin{equation*}
		\begin{split}
				\{&1_{(k,N-k)}, \ e_{r}1_{(k,N-k)}, \ f_{s}1_{(k,N-k)},\ (\psi^{+})^{\pm 1}1_{(k,N-k)}, \ (\psi^{-})^{\pm 1}1_{(k,N-k)},\ h_{\pm 1}1_{(k,N-k)} \ | \\
				&\ 0 \leq k \leq N, \ -k-1 \leq r \leq 0, \ 0 \leq s \leq N-k+1 \},
		\end{split}
	\end{equation*}
	with the following relations
	\begin{equation} \tag{U01}   \label{U01}
	\begin{split}
		&1_{(k,N-k)}1_{(l,N-l)}=\delta_{k,l}1_{(k,N-k)}, \ e_{r}1_{(k,N-k)}=1_{(k-1,N-k+1)}e_{r}, \   f_{r}1_{(k,N-k)}=1_{(k+1,N-k-1)}f_{r}, \\ &(\psi^{+})^{\pm 1}1_{(k,N-k)}=1_{(k,N-k)}(\psi^{+})^{\pm 1},\
		h_{\pm 1}1_{(k,N-k)}=1_{(k,N-k)}h_{ \pm 1},	
	\end{split}
	\end{equation}
	\begin{equation}\tag{U02}   \label{U02}
	\{(\psi^{+})^{\pm 1}1_{(k,N-k)},(\psi^{-})^{\pm 1}1_{(k,N-k)},h_{\pm 1}1_{(k,N-k)}\ | \ 0 \leq k \leq N \} \ \mathrm{pairwise\ commute},
	\end{equation}
	\begin{equation} \tag{U03} \label{U03}
	(\psi^{+})^{\pm 1} \cdot (\psi^{+})^{\mp 1} 1_{(k,N-k)} = 1_{(k,N-k)}=(\psi^{-})^{\pm 1} \cdot (\psi^{-})^{\mp 1}1_{(k,N-k)},
	\end{equation}
	\begin{equation} \tag{U04} \label{U04}
	e_{r}e_{s}1_{(k,N-k)}=-e_{s+1}e_{r-1}1_{(k,N-k)},
	\end{equation}
	\begin{equation}\tag{U05} \label{U05}
	f_{r}f_{s}1_{(k,N-k)}=-f_{s-1}f_{r+1}1_{(k,N-k)},
	\end{equation}
	\begin{equation} \tag{U06} \label{U06}
	\psi^{\pm}e_{r}1_{(k,N-k)}=
	-e_{r+1}\psi^{\pm}1_{(k,N-k)},
	\end{equation}
	\begin{equation} \tag{U07} \label{U07}
	\psi^{\pm}f_{r}1_{(k,N-k)}=
	-f_{r-1}\psi^{\pm}1_{(k,N-k)},
	\end{equation}
	\begin{align*} \tag{U08} \label{U08} 
	&[h_{\pm 1},e_{r}]1_{(k,N-k)}=0 
	&[h_{\pm 1},f_{r}]1_{(k,N-k)}=0,
	\end{align*}
	
	\begin{equation} \tag{U09} \label{U09} 
	 [e_{r},f_{s}]1_{(k,N-k)}= \begin{cases}
	\psi^{+}h_{1}1_{(k,N-k)} & \text{if} \  r+s=N-k+1 \\
	\psi^{+}1_{(k,N-k)} & \text{if} \ r+s=N-k \\
	0 & \text{if} \  -k+1 \leq r+s \leq N-k-1 \\
	-\psi^{-}1_{(k,N-k)} & \text{if} \ r+s=-k \\
	-\psi^{-}h_{-1}1_{(k,N-k)} & \text{if} \ r+s=-k-1
	\end{cases},
	\end{equation}
	for any  $r,s$ such that the above relations make sense.
\end{definition}

\subsection{The categorical action}

In this section, we recall the definition of the categorical action of $\dot{\Uu}_{0,N}(L\SL_2)$.

\begin{definition}  \label{definition 2}
	A categorical $\dot{\Uu}_{0,N}(L\SL_2)$ action consists of a target 2-category $\Kk$, which is triangulated, $\CC$-linear and idempotent complete. The objects in $\Kk$ are
	\[
	\mathrm{Ob}(\Kk)=\{\Kk(k,N-k)\ |\ 0 \leq k \leq N \}
	\] where each $\Kk(k,N-k)$ is also a triangulated category, and each Hom space $\mathrm{Hom}(\Kk(k,N-k),\Kk(l,N-l))$ is also triangulated. 
	
	The morphisms are given by 
	\begin{enumerate}
		\item 1-morphisms:
		$\bo_{(k,N-k)}$, ${\E}_{r}\bo_{(k,N-k)}=\bo_{(k-1,N-k+1)}{\E}_{r}$, ${\F}_{s}\bo_{(k,N-k)}=\bo_{(k+1,N-k-1)}{\F}_{s}$, $({\spi}^{\pm})^{\pm 1}\bo_{(k,N-k)}=\bo_{(k,N-k)}({\spi}^{\pm})^{\pm 1}$, ${\Hhh}_{\pm1}\bo_{(k,N-k)}=\bo_{(k,N-k)}{\Hhh}_{\pm 1}$,  where $-k-1 \leq r \leq 0$, $0 \leq s \leq N-k+1$. Here $\bo_{(k,N-k)}$ is the identity functor of $\Kk(k,N-k)$.
	\end{enumerate}
	
	Subject to the following relations.
	
	\begin{enumerate}
		\item The space of maps between any two 1-morphisms is finite dimensional.
		\item ${\Hhh}_{\pm1}\bo_{(k,N-k)}$ are adjoint to each other, i.e. $({\Hhh}_{1}\bo_{(k,N-k)})^{L} \cong \bo_{(k,N-k)}{\Hhh}_{-1}  \cong ({\Hhh}_{1}\bo_{(k,N-k)})^{R}$.
		\item The left and right adjoints of $\E_{r}$ and $\F_{s}$ are given by conjugation of $\spi^{\pm}$ up to homological shifts. More precisely,
		\begin{enumerate}
			\item $({\E}_{r}\bo_{(k,N-k)})^{R} \cong \bo_{(k,N-k)}{({\spi}^{+})^{r+1}}{\F}_{i,N-k+2}({\spi}^{+})^{-r-2}[-r-1]$
			\item $({\E}_{r}\bo_{(k,N-k)})^{L} \cong \bo_{(k,N-k)}({\spi}^{-})^{r+k-1}{\F}({\spi}^{-})^{-r-k}[r+k]$
			\item $({\F}_{s}\bo_{(k,N-k)})^{R} \cong \bo_{(k,N-k)}({\spi}^{-})^{-s+1}{\E}_{-k-2}({\spi}^{-})^{s-2}[s-1]$ 
			\item $({\F}_{s}\bo_{(k,N-k)})^{L} \cong \bo_{(k,N-k)}({\spi}^{+})^{-s+N-k-1}{\E}({\spi}^{+})^{s-N+k}[-s+N-k]$.
		\end{enumerate}
		\item $ ({\spi}^{\pm})^{\pm 1} ({\spi}^{\pm})^{\pm 1} \bo_{(k,N-k)} \cong ({\spi}^{\pm})^{\pm 1}({\spi}^{\pm})^{\pm 1} \bo_{(k,N-k)}$.
		\item ${\Hhh}_{\pm 1} {\Hhh}_{\pm 1} \bo_{(k,N-k)}  \cong  {\Hhh}_{\pm 1} {\Hhh}_{\pm 1} \bo_{(k,N-k)}$.
		\item ${\Hhh}_{\pm 1} {\spi}^{\pm} \bo_{(k,N-k)} \cong {\spi}^{\pm} {\Hhh}_{\pm 1} \bo_{(k,N-k)}$.
		\item The relations between ${\E}_{r}$, ${\E}_{s}$ are given by the following
			\[
			{\E}_{r+1}{\E}_{s}\bo_{(k,N-k)} \cong \begin{cases}
			{\E}_{s+1}{\E}_{r}\bo_{(k,N-k)}[-1] & \text{if} \ r-s \geq 1 \\
			0 & \text{if} \ r=s \\
			{\E}_{s+1}{\E}_{r}\bo_{(k,N-k)}[1] & \text{if} \ r-s \leq -1. 
			\end{cases}
			\]
		\item The relations between ${\F}_{r}$, ${\F}_{s}$ are given by the following
			We have
			\[
			{\F}_{r}{\F}_{s+1}\bo_{(k,N-k)} \cong \begin{cases}
			{\F}_{s}{\F}_{r+1}\bo_{(k,N-k)}[1] & \text{if} \ r-s \geq 1 \\
			0 & \text{if} \ r=s \\
			{\F}_{s}{\F}_{r+1}\bo_{(k,N-k)}[-1] & \text{if} \ r-s \leq -1. 
			\end{cases}  
			\]

		\item The relations between ${\E}_{r}$, $\Psi^{\pm}$ are given by the following. We have 
			\[
			{\spi}^{\pm}{\E}_{r}\bo_{(k,N-k)} \cong {\E}_{r+1}{\spi}^{\pm}\bo_{(k,N-k)}[\mp 1].
			\]

		\item The relations between ${\F}_{r}$, $\Psi^{\pm}$ are given by the following. We have 
			\[
			{\spi}^{\pm}{\F}_{r}\bo_{(k,N-k)} \cong {\F}_{r-1}{\spi}^{\pm}\bo_{(k,N-k)}[\pm 1].
			\]

		\item The relations between ${\E}_{r}$, $\Hhh_{\pm 1}$ are given by the following. They are related by the following exact triangles
			\[
			{\Hhh}_{1}{\E}_{r}\bo_{(k,N-k)} \rightarrow {\E}_{r}{\Hhh}_{1}\bo_{(k,N-k)} \rightarrow ({\E}_{r+1}\bigoplus {\E}_{r+1}[1])\bo_{(k,N-k)}, 
			\]
			\[
			{\E}_{r}{\Hhh}_{-1}\bo_{(k,N-k)} \rightarrow {\Hhh}_{-1}{\E}_{r}\bo_{(k,N-k)} \rightarrow   ({\E}_{r-1}\bigoplus {\E}_{r-1}[1])\bo_{(k,N-k)}.
			\]
		
		\item The relations between ${\F}_{r}$, $\Hhh_{\pm 1}$ are given by the following. They are related by the following exact triangles
			\[
			{\F}_{r}{\Hhh}_{1}\bo_{(k,N-k)} \rightarrow {\Hhh}_{1}{\F}_{r}\bo_{(k,N-k)} \rightarrow  ({\F}_{r+1}\bigoplus {\F}_{r+1}[1])\bo_{(k,N-k)},
			\]
			\[
			{\Hhh}_{-1}{\F}_{r}\bo_{(k,N-k)} \rightarrow {\F}_{r}{\Hhh}_{-1}\bo_{(k,N-k)} \rightarrow ({\F}_{r-1}\bigoplus {\F}_{r-1}[1])\bo_{(k,N-k)} .
			\]
			
		\item For ${\E}_{r}{\F}_{s}\bo_{(k,N-k)}, {\F}_{s}{\E}_{r}\bo_{(k,N-k)} \in \mathrm{Hom}(\Kk((k,N-k)),\Kk((k,N-k)))$, they are related by  exact triangles, more precisely, 
		\begin{enumerate}
			\item
			\[
			{\F}_{s}{\E}_{r}{\bo}_{(k,N-k)} \rightarrow {\E}_{r}{\F}_{s}{\bo}_{(k,N-k)} \rightarrow {\spi}^{+}{\Hhh}_{1}{\bo}_{(k,N-k)} \ \text{if} \ r+s=N-k+1,
			\]
			\item \[
			{\F}_{s}{\E}_{r}\bo_{(k,N-k)} \rightarrow {\E}_{r}{\F}_{s}\bo_{(k,N-k)} \rightarrow {\spi}^{+}\bo_{(k,N-k)} 
			\ \text{if} \  
			r+s =N-k,
			\] 
			\item \[
			{\E}_{r}{\F}_{s}\bo_{(k,N-k)} \rightarrow {\F}_{s}{\E}_{r}\bo_{(k,N-k)} \rightarrow {\spi}^{-}\bo_{(k,N-k)}   \ \text{if} \  r+s = -k,
			\] 
			\item \[
			{\E}_{r}{\F}_{s}\bo_{(k,N-k)} \rightarrow {\F}_{s}{\E}_{r}\bo_{(k,N-k)} \rightarrow {\spi}^{-}{\Hhh_{-1}}\bo_{(k,N-k)}  \ \text{if} \  r+s = -k-1 ,
			\] 
			\item \[
			{\F}_{s}{\E}_{r}\bo_{(k,N-k)} \cong {\E}_{r}{\F}_{s}\bo_{(k,N-k)} \ \text{if} \ -k+1 \leq r+s \leq N-k-1 .
			\] 
		\end{enumerate}

	\end{enumerate}
	
\end{definition}

\subsection{Geometric example}

In this section, we recall the main theorem (Theorem 5.2) in \cite{Hsu}, which says that there is a categorical action of $\dot{\Uu}_{0,N}(L\SL_2)$ on the bounded derived categories of coherent sheaves on Grassmannians.

We denote $\Dd^b(\GG(k,N))$ to be the bounded derived categories of coherent sheaves on $\GG(k,N)$. Those would be the objects $\Kk(k,N-k)$ of the triangulated $2$-category $\Kk$ in Definition \ref{definition 2}. On $\GG(k,N)$ we denote $\V$ to be the tautological bundle of rank $k$.

To define those $1$-morphisms  ${\E}_{r}\bo_{(k,N-k)}$, ${\F}_{s}\bo_{(k,N-k)}$, ${\Hhh}_{\pm 1}\bo_{(k,N-k)}$, $({\spi}^{\pm})^{\pm 1}\bo_{(k,N-k)}$, we use the language of FM transforms, that means we would define them by using FM kernels. So we have to introduce more geometries.

We define correspondences via 3-steps partial flag variety $Fl(k-1,k) \subset \GG(k,N) \times \GG(k-1,N)$ by
\[
Fl(k-1,k):=\{0 \overset{k-1}{\subset} V' \overset{1}{\subset} V \overset{N-k}{\subset} \CC^N \}.
\]
Then we have the natural line bundle $\V/\V'$ on $Fl(k-1,k)$.

Next, we introduce new varieties
\begin{equation*}
\begin{split}
X(k):=\{& (V''',V,V',V') \in \GG(k-1,N) \times \GG(k,N) \times \GG(k,N) \times \GG(k+1,N) \ | \ \\
& V'''\subset V \subset V', \
V'''\subset V'' \subset V' \}.
\end{split}
\end{equation*}

On $X(k)$, we have the divisor $D(k) \subset X(k)$ that defined by
\begin{equation*}
\begin{split}
D(k):=\{& (V''',V,V'',V') \in \GG(k-1,N) \times \GG(k,N) \times \GG(k,N) \times \GG(k+1,N)  \ | \ \\
& V'''\subset V=V'' \subset V' \}
\end{split}
\end{equation*} 
which is cut out by the natural section of the line bundles $\Hh om(\V''/\V''',\V'/\V)$ or $\Hh om(\V/\V''',\V'/\V'')$. More precisely, we have $\Oo_{X(k)}(-D(k)) \cong \V''/\V''' \otimes (\V'/\V)^{-1}\cong \V/\V''' \otimes (\V'/\V'')^{-1}$.

Let $p(k):X(k) \rightarrow Y(k)$ be the projection by forgetting $V'''$ and $V'$. Here 
\begin{equation*}
Y(k)=\{(V,V'') \in  \GG(k,N) \times \GG(k,N) \ | \
\dim V \cap V'' \geq k-1\}.
\end{equation*}

Let $t(k):Y(k) \rightarrow \GG(k,N) \times \GG(k,N)$ be the inclusion and $\Delta(k):\GG(k,N) \rightarrow \GG(k,N) \times \GG(k,N)$  the diagonal map, and $\iota(k):Fl(k-1,k) \hookrightarrow \GG(k,N) \times \GG(k-1,N)$ be the inclusion. Then we define those 1-morphisms via using the above geometries.

\begin{definition} \label{definition 3}
	We define  $\bo_{(k,N-k)}$, ${\E}_{r}\bo_{(k,N-k)}$, $\bo_{(k,N-k)}{\F}_{s}$, ${\Hhh}_{\pm 1}\bo_{(k,N-k)}$, $({\spi}^{\pm})^{\pm 1}\bo_{(k,N-k)}$ to be FM transforms with the corresponding kernels
	\begin{equation*}
	\begin{split}
	\bo_{(k,N-k)}&:= \Delta(k)_{*}\Oo_{\GG(k,N)} \in \Dd^b(\GG(k,N)\times \GG(k,N)), \\
	\Ee_{r}\bo_{(k,N-k)}&:=\iota(k)_{*} (\V/\V')^{r} \in \Dd^b(\GG(k,N)\times \GG(k-1,N)), \\
	\bo_{(k,N-k)}\Ff_{r}&:=\iota(k)_{*} (\V'/\V)^{r} \in \Dd^b(\GG(k-1,N)\times \GG(k,N)), \\
	(\Psi^{+})^{\pm 1}\bo_{(k,N-k)}&:=\Delta(k)_{*}\tdet (\CC^N/\V)^{\pm 1}[\pm(1+k-N)] \in \Dd^b(\GG(k,N) \times \GG(k,N)), \\
	(\Psi^{-})^{\pm 1}\bo_{(k,N-k)}&:=\Delta(k)_{*} \tdet (\V)^{\mp 1}[\pm(1-k)] \in \Dd^b(\GG(k,N)\times \GG(k,N)), \\
	\Hh_{1}\bo_{(k,N-k)}&:=(\Psi^{+}\bo_{(k,N-k)})^{-1} \ast [t(k)_{*}p(k)_{*}(\Oo_{2D(k)} \otimes (\V'/\V)^{N-k+1})] \in \Dd^b(\GG(k,N) \times \GG(k,N)), \\
	\Hh_{-1}\bo_{(k,N-k)}&:=(\Psi^{-}\bo_{(k,N-k)})^{-1} \ast [t(k)_{*}p(k)_{*}(\Oo_{2D(k)} \otimes (\V/\V''')^{-k-1})] \in \Dd^b(\GG(k,N)\times \GG(k,N)).
	\end{split}
	\end{equation*}
\end{definition}

Then we have the following theorem.

\begin{theorem} [$\SL_{2}$ version of Theorem 5.2 \cite{Hsu}] \label{Theorem 3}
	Let $\Kk$ be the triangulated 2-categories whose nonzero objects are $\Kk(k,N-k)=\Dd^b(\GG(k,N))$, the 1-morphisms are kernels defined in Definition \ref{definition 3} and the 2-morphisms are maps between kernels.  Then this gives a categorical $\dot{\Uu}_{0,N}(L\SL_2)$ action.
\end{theorem}

\section{t-structures via  exceptional collections} \label{Section 4}

In this section, we recall the definition and some basic properties of t-structures.
We also give the tools for dealing with t-structures that would be used in the next section. Then we mention a result of R. Bezrukavnikov for how  to construct a t-structure from an exceptional collections. Finally, we apply this to the Kapranov exceptional collection constructed in Section \ref{Section 2}.

\subsection{t-structures and related tools}
Let us recall the definition of t-structures. Fix $\Dd$ be a ($\CC$-linear) triangulated category.

\begin{definition} \label{Definition 1} 
	A t-structure on $\Dd$ is a pair ($\Dd^{\leq 0}$, $\Dd^{\geq 0}$) of two full subcategories of $\Dd$ such that 
	\begin{enumerate}
		\item $\Dd^{\geq 1} \subseteq \Dd^{\geq 0}$, $\Dd^{\leq -1} \subseteq \Dd^{\leq 0}$, where we define $\Dd^{\leq n} := \Dd^{\leq 0}[-n]$ and $\Dd^{\geq n} := \Dd^{\geq 0}[-n]$. 
		\item $\Hom_{\Dd}(X,Y)=0$ for all $X \in \text{Ob}(\Dd^{\leq 0})$ and $Y \in \text{Ob}(\Dd^{\geq 1})$.
		\item For all $X \in \text{Ob}(\Dd)$, there is an exact triangle 
		$Y \rightarrow X \rightarrow Z \rightarrow Y[1]$ where $Y \in \text{Ob}(\Dd^{\leq 0})$ and $Z \in \text{Ob}(\Dd^{\geq 1})$.
	\end{enumerate}
\end{definition}

A t-structure on $\Dd$ gives rise to an abelian category.

\begin{definition}
	Given a t-structure ($\Dd^{\leq 0}$, $\Dd^{\geq 0}$) on $\Dd$, the heart (core) is the intersection $\Dd^{\ensuremath\heartsuit}:=\Dd^{\leq 0} \cap \Dd^{\geq 0}$.
\end{definition}

Let $\iota^{\leq n} : \Dd^{\leq n} \rightarrow \Dd$ and $\iota^{\geq n} : \Dd^{\geq n} \rightarrow \Dd$ be the inclusion functors. Then $\iota^{\leq n}$ has a right adjoint $\tau^{\leq n} :\Dd \rightarrow \Dd^{\leq n}$. Similarly, $\iota^{\geq n}$ has a left adjoint $\tau^{\geq n} :\Dd \rightarrow \Dd^{\geq n}$. Then we have the cohomological functor 
\[ 
H^{n} =  \tau^{\leq n} \circ \tau^{\geq n} = \tau^{\geq n} \circ \tau^{\leq n} : \Dd \rightarrow \Dd^{\heartsuit}.
\]

We give the definitions for the notions of weakly generate and conservative. 

\begin{definition}
Let $\Cc$ and $\Dd$ be triangulated categories. We say that a collection of objects $\{X_{s} :  s \in S \}$ weakly generate $\Cc$ if $\Hom(X_{s}[n],Y)=0$ for all $s \in S$ and $n \in \ZZ$ implies $Y=0$. A functor $\Phi:\Cc \rightarrow \Dd$ is called conservative if $\Phi(X)=0$ implies $X=0$.
\end{definition}

The following result will be the main tool for us to  construct t-structures in the following section.

\begin{theorem}[\cite{Po}, Theorem 2.1.2,] \label{Theorem 8} 
	Let $\Phi:\widetilde{\Cc} \rightarrow \widetilde{\Dd}$ be a functor of cocomplete triangulated categories which commute with all small coproducts and that admits a left adjoint functor $\Phi^{L}:\widetilde{\Dd} \rightarrow \widetilde{\Cc}$. Let $\Cc \subset \widetilde{\Cc}$ and $\Dd \subset \widetilde{\Dd}$ be full subcategories such that $\Phi(\Cc) \subset \Dd$, $\Phi^{L}(\Dd) \subset \Cc$ and $\Phi(X) \in \Dd \Rightarrow X \in \Cc$ for any $X \in \text{Ob}(\widetilde{\Cc})$.
	
	Assume that $\Dd$ has a t-structure $(\Dd^{\leq 0}, \Dd^{\geq 0})$ such that $\Phi \circ \Phi^{L} : \Dd \rightarrow \Dd$ is right t-exact. Then there exists a unique t-structure on $\Cc$ with 
	\[
	\Cc^{\geq 0} = \{X \in \text{Ob}(\Cc) \ | \ \Phi(X) \in \text{Ob}(\Dd^{\geq 0}) \}.
	\] 
	Moreover, the functor $\Phi$ is t-exact with respect to this t-structure.
\end{theorem}

We will always assume that whenever we want to apply Theorem \ref{Theorem 8} to a functor $F:\Cc \rightarrow \Dd$ between triangulated categories, we can extend it to a functor $\tilde{F}:\tilde{\Cc} \rightarrow \tilde{\Dd}$ between cocomplete triangulated categories which satisfies all the conditions in the statement of Theorem \ref{Theorem 8}. 
	
The reason for this is that in all the examples we consider the categories $\Cc$ and $\Dd$ will be bounded derived categories of coherent sheaves $\Dd^b(X)$ and $\Dd^b(Y)$ for smooth varieties $X$ and $Y$. We take $\widetilde{\Cc}$ and $\widetilde{\Dd}$ to be the corresponding cocomplete unbounded derived categories of quasi-coherent sheaves $\Dd_{qcoh}(X)$ and $\Dd_{qcoh}(Y)$.

Any functor $\Phi:\Cc \rightarrow \Dd$ will be given by a FM kernel, and thus naturally extends to quasi-coherent sheaves. Furthermore, all such functors will have right adjoints and commute with small coproducts by the Adjoint functor theorem. Finally, $\Phi$ will always be conservative and by Lemma \ref{lemma 7} below we have $\Phi(X) \in \Dd$ implies $X \in \Cc$.

\begin{lemma} [\cite{CKo}, Lemma 2.3] \label{lemma 7} 
Let $X$ and $Y$ be smooth varieties and $\Phi:\Dd_{qcoh}(X) \rightarrow \Dd_{qcoh}(Y)$ be a functor that given by a FM kernel in $\Dd^b(X \times Y)$. If the restriction of $\Phi$ to $\Dd^b(X)$ is conservative, then for any $\Ff \in \Dd_{qcoh}(X)$, $\Phi(\Ff) \in \Dd^b(Y)$ if and only if $\Ff \in \Dd^b(X)$.
\end{lemma} 

We collect some technical results that we will need about t-structures.

\begin{lemma} [\cite{Po}, Lemma 1.1.1] \label{lemma 8}
Let $\Phi: \Cc_{1} \rightarrow \Cc_{2}$ be a conservative t-exact functor between triangulated categories. Then 
\begin{align*}
	&\Cc_{1}^{\leq 0}=\{X\in \text{Ob}(\Cc_{1}) \ | \ \Phi(X) \in \text{Ob}(\Cc_{2}^{\leq 0}) \} \\
	&\Cc_{1}^{\geq 0}=\{X\in \text{Ob}(\Cc_{1}) \ | \ \Phi(X) \in \text{Ob}(\Cc_{2}^{\geq 0}) \} 
\end{align*}
\end{lemma}

\begin{lemma}[\cite{CKo}, Lemma 2.6] \label{lemma 9} 
Let $(\Cc^{\leq 0}, \Cc^{\geq 0})$ be a t-structure on $\Cc$ and suppose $X, Y \in \Cc$. If $X \oplus Y \in \Cc^{\leq 0}$ then $X, Y \in \Cc^{\leq 0}$. Similarly, if $X \oplus Y \in \Cc^{\geq 0}$ then $X, Y \in \Cc^{\geq 0}$.
\end{lemma}

\begin{lemma}[\cite{CKo}, Lemma 2.8] \label{lemma 10} 
	Suppose we have two t-structures $(\Cc^{\leq_{1} 0}, \Cc^{\geq_{1} 0})$ and $(\Cc^{\leq_{2} 0}, \Cc^{\geq_{2} 0})$ on $\Cc$ such that $\Cc^{\leq_{1} 0} \subset \Cc^{\leq_{2} 0}$ and $\Cc^{\geq_{1} 0} \subset \Cc^{\geq_{2} 0}$. Then the two t-structures are identical.
\end{lemma}

\subsection{A result due to R. Bezrukavnikov}
One of the applications for exceptional collections is to construct t-structures on triangulated categories. In order to construct t-structures on $\Dd$, we need the notion of dual exceptional collections.

\begin{definition} \label{definition 8} 
 	Let $\{E_{1},...,E_{n}\}$ be an exceptional collection on the triangulated category $\Dd$. A set $\{F_{1},...,F_{n}\}$ of objects in $\Dd$ is called dual to $\{E_{1},...,E_{n}\}$ if $\Hom_{\Dd}(F_{i},E_{i})=\CC$ and $\Hom_{\Dd}(F_{i},E_{j})=0$ for $i \neq j$.
\end{definition} 

\begin{remark}
	The set $\{F_{1},...,F_{n}\}$ is exceptional provided that $\{E_{1},...,E_{n}\}$ is, where the order on $\{F_{1},...,F_{n}\}$ is opposite to that on $\{E_{1},...,E_{n}\}$.
\end{remark}

The following result is due to R. Bezrukavnikov.

\begin{theorem}[\cite{Be1} Proposition 1, \cite{Be2} Proposition 2] \label{Theorem 6}
	Let $\Dd$ be a triangulated category. Let $\{E_{1},...,E_{n}\}$ be a full exceptional collection and $\{F_{1},...,F_{n}\}$  be its dual exceptional collection. Then there exists an unique t-structure ($\Dd^{\leq 0}$, $\Dd^{\geq 0}$) on $\Dd$ such that $E_{i} \in \text{Ob}(\Dd^{\geq 0})$ and $F_{i} \in \text{Ob}(\Dd^{\leq 0})$. Moreover, $\Dd^{\leq 0}$, $\Dd^{\geq 0}$ are given by
		\begin{align*}
			& \Dd^{\geq 0}=\langle \{E_{i}[d] \ | \ 1 \leq i \leq n, \ d \leq 0 \} \rangle \\
			& 	\Dd^{\leq 0}=\langle \{F_{i}[d] \ | \ 1 \leq i \leq n, \ d \geq 0 \} \rangle
		\end{align*}
\end{theorem}

\begin{remark}
For a set $S$ of objects in $\Dd$, we denote $\langle S \rangle$ to be the full subcategory of $\Dd$ generated by $S$ under extensions.
\end{remark}

\begin{remark}
The above theorem is part of R. Bezrukavnikov's result in \cite{Be1}, \cite{Be2}. In fact, R. Bezrukavnikov also give a description of the irreducible objects in the heart $\Dd^{\heartsuit}$ of the above t-structure. However, the author would not touch this in this article, so we just state what we need here.
\end{remark}

\subsection{Application to t-structures on $\Dd^b(\GG(k,N))$}

From Corollary \ref{Corollary 2}, we have the  exceptional collections $\R_{(k,N-k)}$  for $\Dd^b(\GG(k,N))$. So it would be natural to apply Theorem \ref{Theorem 6} to construct t-structures on $\Dd^b(\GG(k,N))$.

Since we only have the exceptional collection, in order to get a t-structure, we need to find the dual exceptional collection. The following lemma would be useful for us to find the dual exceptional collection.

\begin{lemma} [\cite{Kap88} Lemma 3.2]  \label{lemma 3}  
	If $N-k \geq \lambda_{1} \geq \lambda_{2} \geq ...\geq \lambda_{k} \geq 0$ and $k \geq \mu_{1} \geq \mu_{2} \geq ... \geq \mu_{N-k} \geq 0$. Then one has 
	\begin{equation*}
		H^{i}(\GG(k,N),\s_{\lambda}\V \otimes \s_{\mu} (\CC^N/\V)^{\vee})= \begin{cases}
			0 & \text{for all} \ i, \ \text{if} \ \lambda \neq \mu^{*} \\
			0 &  \text{for} \ i \neq |\lambda| = |\mu|, \ \text{if} \ \lambda = \mu^{*} \\
			\CC &  \text{for} \ i = |\lambda| = |\mu|, \ \text{if} \ \lambda = \mu^{*} 
		\end{cases}
	\end{equation*}
\end{lemma}

For the exceptional collection $\R_{(k,N-k)}=\{\s_{\lambda}\V \ | \ \lambda \in P(N-k,k) \}$, applying Lemma \ref{lemma 3}, it is easy to see that the dual exceptional collection for $\R_{(k,N-k)}$ is given by 
\begin{equation} \label{eq 44} 
\R'_{(k,N-k)} = \{\s_{\mu}(\CC^N/\V)[-|\mu|] \ | \ \mu \in P(k,N-k) \}.
\end{equation}

Thus we obtain the following t-structure on $\Dd^b(\GG(k,N))$ by using Theorem \ref{Theorem 6}.

\begin{corollary} \label{Corollary 4}
	There is a t-structure $(\Dd_{ex}^{\leq 0}(k,N-k), \Dd_{ex}^{\geq 0}(k,N-k))$  on $\Dd^b(\GG(k,N))$ such that $\s_{\lambda}\V \in \text{Ob}(\Dd_{ex}^{\geq 0}(k,N-k))$ and $\s_{\mu}(\CC^N/\V)[-|\mu|] \in \text{Ob}(\Dd_{ex}^{\leq 0}(k,N-k))$. Moreover 
	\begin{align*}
		&\Dd_{ex}^{\geq 0}(k,N-k)=\langle \{\s_{\lambda}\V[d] \ | \ \lambda \in P(N-k,k), \ d \leq 0 \} \rangle \\
		& \Dd_{ex}^{\leq 0}(k,N-k)=\langle \{\s_{\mu}(\CC^N/\V)[-|\mu|+d]  \ | \ \mu \in P(k,N-k), \ d \geq 0 \} \rangle.
	\end{align*}
\end{corollary}

\section{t-structure via categorical action} \label{Section 5}

In this section,  we prove the main theorem of this article. Roughly speaking, given a categorical action of the shifted $q=0$ affine algebra and a fixed t-structure on the highest height category, we use this action to translate this t-structure to other weight categories.

From Theorem \ref{Theorem 3}, we have categorical action of $\dot{\Uu}_{0,N}(L\SL_{2})$ on $\bigoplus_{k} \Dd^b(\GG(k,N))$. So we can apply this main theorem to get t-structures on $\Dd^b(\GG(k,N))$. When the highest weight category is equipped with the standard t-structure, we show that the t-structure we obtain on $\Dd^b(\GG(k,N))$ is precisely the t-structure from  Corollary \ref{Corollary 4}.

Before we move to the main theorem, we need a tool from representations theory which would relate the Kapranov exceptional collection to the categorical action of $\dot{\Uu}_{0.N}(L\SL_{2})$. Moreover, this also serve as a motivation for us to construct the t-structure from categorical action.

\subsection{Borel-Weil-Bott theorem and its relative version} \label{subsection 5.1}

In this section, we recall the classical Borel-Weil-Bott theorem (for type A) and extend it to a relative version. We refer to \cite{D} and \cite{Lur} for more details.

Let $r \geq 2$ be a positive integer. Consider the full flag variety
\[
Fl(r)=\{0=V_{0} \subset V_1 \subset V_2 \subset ... \subset \CC^r=V_{r} \ | \ \dim V_{i}=i \}.
\]

Recall some basic stuffs for the theorem. Let $G=\GLL_{r}(\CC)$ (or $\SLL_{r}(\CC)$) and choose a Borel subgroup $B$ (invertible upper triangular  matrices)  with a maximal torus $T$ (invertible diagonal matrices)  such that $T \subset B \subset G$ and $G/B=Fl(r)$. Also, let $U \subset B$ be the unipotent subgroup such that $B=TU$. A character of $B$ is a morphism $\lambda:B \rightarrow \CC^*$, it is a one-dimensional representation of $B$. Let $B$ act on $G \times \CC$ by $b(g,t)=(gb^{-1},\lambda(b)t)$ for all $b \in B$ and $(g,t) \in G\times \CC$.

This action is free and we define the quotient $L_{\lambda}=G \times_{B} \CC:= (G \times \CC)/B$. The map $L_{\lambda} \rightarrow G/B$ defined by $(g,t)B \mapsto gB$ is a equivariant line bundle on the full flag variety $G/B$. 

Note that since $B=TU$, any such character $\lambda$ is uniquely determined by its restriction to $T$. Furthermore, for $T$ is the group of invertible diagonal matrices, we can see that each $\lambda:T \rightarrow \CC^*$ is of the form
\[
diag(a_1,...,a_r) \mapsto a_1^{\lambda_1}...a_{r}^{\lambda_{r}}
\] for some $\lambda_{1},...,\lambda_{r} \in \ZZ$. We denote $\Lambda(T)$ the multiplicative group of characters of $T$ (also called weights). The above identifies $\Lambda(T)$ to the additive group $\ZZ^r$.

Thus, a weight $\lambda:T \rightarrow \CC^*$  corresponds to a $r$-tuple of integers $(\lambda_{1},...,\lambda_{r})$, and we have the following definition.

\begin{definition}
	We say that a weight $\lambda$ is dominant (resp. regular dominant) if $\lambda_{1} \geq ... \geq \lambda_{r}$ (resp. $\lambda_{1}>...>\lambda_{r}$).
\end{definition}

\begin{definition}
	The fundamental weights are $\chi_{1},\chi_{2},...,\chi_{r-1}$ such that 
	\[
	\chi_{i}=(1,...,1 \ (i\ \text{times}),0,..,0 \ (r-i \ \text{times})).
	\]
	Let $\rho=\sum_{i=1}^{r-1} \chi_{i}=(r-1,r-2,..,1,0)$ this is the half-sum of positive roots.
\end{definition}

Denote $W$ to be the Weyl group corresponds to the pair $(G,T)$, which is juts the symmetric group $S_{r}$. Then $W$ acts on $\Lambda(T)$ in the following way. For any element $w \in W$, we define $w \star \lambda := w(\lambda+\rho)-\rho$.  Also, for any $w \in W$, the length of $w$ is defined to be 
\[
l(w):=\# \{(i,j)|1 \leq i < j \leq r, w(i) > w(j)\}.
\]

We let $w_{0}$ to be the reverse permutation, i.e. $w_{0}(i)=r-i+1$ for all $1 \leq i \leq r$. It has the longest length $r(r-1)/2$, which is equal to the dimension of the full flag variety $Fl(r)$.

Then we can state the Borel-Weil-Bott theorem (for type A).

\begin{theorem}(Borel-Weil-Bott) \label{Theorem 4}
	Let $\lambda:T \rightarrow \CC^*$ be a weight, and $L_{\lambda}$ be the associated equivariant line bundle on $G/B$.  Then we have the following cases \\
	(1) If $-\lambda$ is dominant, then we have $H^0(G/B,L_{\lambda})$ is an irreducible representation of $G$ with lowest weight $\lambda$ (and therefore dual to the irreducible representation of $G$ with highest weight $-\lambda$), i.e. $H^0(G/B,L_{\lambda}) \simeq (\s_{-\lambda}\CC^r)^{\vee}$. Also, $H^i(G/B,L_{\lambda})=0$ for all $i \neq 0$. \\
	(2)If $-\lambda$ is not dominant, and there is no $w \in W$ such that $w \star (-\lambda)$ is dominant, then $H^i(G/B,L_{\lambda})=0$ for all $i$. \\
	(3)If $-\lambda$ is not dominant, and there is an unique $w \in W$ such that $w \star (-\lambda)$ is dominant, then $H^{l(w)}(G/B,L_{\lambda})$ is dual to the irreducible representation of $G$ with highest weight $w \star (-\lambda)$, i.e. $H^{l(w)}(G/B,L_{\lambda}) \simeq (\s_{w \star (-\lambda)}\CC^r)^{\vee}$. All other cohomologies vanish.
\end{theorem}

\subsubsection{Generalization}
Next, we can generalize the Theorem \ref{Theorem 4}  to a relative version. Note that from Theorem \ref{Theorem 3}, we have a categorical action of $\dot{\Uu}_{0,N}(L\SL_2)$ on $\bigoplus_{k} \Dd^b(\GG(k,N))$. 

The action of $\dot{\Uu}_{0,N}(L\SL_2)$ can be described as the following picture.
\[
\xymatrix{ 
	.... \Dd^b(\GG(k+1,N))   \ar@/^/[r]^{{\E}_{r}}   
	& \Dd^b(\GG(k,N))   \ar@/^/[r]^{{\E}_{r}}   \ar@/^/[l]^{{\F}_{s}}
	& \Dd^b(\GG(k-1,N)) \ar@/^/[l]^{{\F}_{s}} ....    }
\] for various $r,s$ within the range in Definition \ref{definition 1}.

More generally, we can consider $r,s \in \ZZ$. Since ${\E}_{r}$ is defined by using the following correspondence
\begin{equation} \label{diagram9}  
	\xymatrix{ 
		&&Fl(k-1,k)=\{0 \overset{k-1}{\subset} V' \overset{1}{\subset} V \overset{N-k}{\subset} \CC^N \} 
		\ar[ld]_{p_1} \ar[rd]^{p_2}   \\
		& \GG(k,N)  && \GG(k-1,N)
	}
\end{equation}  and by definition 
\[
{\E}_{r}:=p_{2*}(p_{1}^{*}\otimes(\V/\V')^{r}):\Dd^{b} (\GG(k,N)) \rightarrow \Dd^{b}(\GG(k-1,N))
\] where $\V,\V'$ to be the tautological bundles on $Fl(k-1,k)$ of rank $k$, $k-1$, respectively. Similarly for ${\F}_{s}$ where $s \in \ZZ$.

Our generalization comes from the study of the composition of functors $\bo_{(k,N-k)}{\F}_{\lambda_{r}}{\F}_{\lambda_{r-1}}...{\F}_{\lambda_{1}}$ (or ${\E}_{\lambda_{1}}{\E}_{\lambda_{2}}...{\E}_{\lambda_{r}}\bo_{(k,N-k)}$ ) and the endomorphism algebra $\End(\bo_{(k,N-k)}{\F}_{\lambda_{r}}{\F}_{\lambda_{r-1}}...{\F}_{\lambda_{1}})$ .

Since we construct the categorical action by using the language of FM transformation, the functors ${\E}_{r}\bo_{(k,N-k)}$ and $\bo_{(k,N-k)}{\F}_{s}$ are FM transformations with kernels $\Ee_{r}\bo_{(k,N-k)}$ and $\bo_{(k,N-k)}\Ff_{s}$ respectively in Definition \ref{definition 3}. 

By Proposition \ref{Proposition 2}, the composition of FM transformations is also a FM transformation, thus the functor $\bo_{(k,N-k)}{\F}_{\lambda_{r}}{\F}_{\lambda_{r-1}}...{\F}_{\lambda_{1}}$ is a FM transformations with kernel $\bo_{(k,N-k)}{\Ff}_{\lambda_{r}} \ast {\Ff}_{\lambda_{r-1}} \ast ...\ast {\Ff}_{\lambda_{1}}$, which is an object in $\Dd^b(\GG(k-r,N) \times \GG(k,N))$. 


To know what exactly is the kernel $\bo_{(k,N-k)} {\Ff}_{\lambda_{r}} \ast {\Ff}_{\lambda_{r-1}} \ast ...\ast {\Ff}_{\lambda_{1}}$ in $\Dd^b(\GG(k-r,N) \times \GG(k,N))$, we use the correspondence
\[
\xymatrix{
	&& Fl(k-r,k-r+1,...,k) \ar[d]^{\pi}  \\
	&& Fl(k-r,k)  \ar[ld]_{q_1} \ar[rd]^{q_2}  \\
	& \GG(k-r,N)  
	&& \GG(k,N)  }
\] where 
\begin{align*}
	&Fl(k-r,k-r+1,...,k)=\{0 \subset V_{k-r} \subset V_{k-r+1} \subset ...\subset V_{k} \subset \CC^N \ | \ \dim V_{i} =i\} \\
	&Fl(k-r,k)=\{0 \subset V_{k-r}\subset V_{k} \subset \CC^N \ | \ \dim V_{i} =i \}
\end{align*} and $\pi$, $q_{1}$, $q_{2}$ are natural projections.

Denote $\V_{i}$ to be the tautological bundle on $Fl(k-r,k-r+1,...,k)$ of rank $i$. Let $\iota_{1}:Fl(k-r,k-r+1,...,k) \hookrightarrow \GG(k-r,N) \times ... \times \GG(k,N)$ and $\iota_{2}: Fl(k-r,k) \hookrightarrow \GG(k-r,N) \times \GG(k,N)$ be the natural inclusions. Also, let $\pi':\GG(k-r,N) \times ... \times \GG(k,N) \rightarrow \GG(k-r,N) \times \GG(k,N)$ be the natural projection. Then we have the following commutative diagram

\begin{equation*} 
	\begin{tikzcd}[column sep=large]
		Fl(k-r,k-r+1,...,k)   \arrow[r, "\iota_{1}"]  \arrow[d, "\pi"] &\GG(k-r,N)\times ... \times \GG(k,N) \arrow[d, "\pi'"]\\
		Fl(k-r,k)  \arrow[r, "\iota_{2}"] & \GG(k-r,N)\times \GG(k,N)
	\end{tikzcd}
\end{equation*}

Thus, by flat base change, we have  
\begin{align*}
	\bo_{(k,N-k)} {\Ff}_{\lambda_{r}} \ast {\Ff}_{\lambda_{r-1}} \ast ...\ast {\Ff}_{\lambda_{1}} 
	&\cong \pi'_{*}\iota_{1*}((\V_{k-r+1}/\V_{k-r})^{\lambda_{1}} \otimes ... \otimes (\V_{k}/\V_{k-1})^{\lambda_{r}}) \\
	&\cong \iota_{2*}\pi_{*}((\V_{k-r+1}/\V_{k-r})^{\lambda_{1}} \otimes ... \otimes (\V_{k}/\V_{k-1})^{\lambda_{r}}) 
\end{align*}

We also have the following fibred product diagram
\begin{equation*} 
	\begin{tikzcd}[column sep=large]
		Fl(r)   \arrow[r]  \arrow[d] & Fl(k-r,k-r+1,...,k) \arrow[d,"\pi"]\\
		pt  \arrow[r] & Fl(k-r,k)
	\end{tikzcd}
\end{equation*} this tells us that each fibre of $\pi$ is a full flag variety isomorphic to $Fl(r)$. Since each $\V_{i}/\V_{i-1}$ is a line bundle, so is $(\V_{k-r+1}/\V_{k-r})^{\lambda_{1}} \otimes ... \otimes (\V_{k}/\V_{k-1})^{\lambda_{r}}$. Thus to calculate $\pi_{*}((\V_{k-r+1}/\V_{k-r})^{\lambda_{1}} \otimes ... \otimes (\V_{k}/\V_{k-1})^{\lambda_{r}})$, we have to calculate the family of sheaf cohomology of the line bundle $(\V_{k-r+1}/\V_{k-r})^{\lambda_{1}} \otimes ... \otimes (\V_{k}/\V_{k-1})^{\lambda_{r}}$ along each fibre of $\pi$.

The first thing we want to do is to relate the line bundle $L_{1}^{\lambda_1}\otimes...\otimes L_{r}^{\lambda_r}$ with he equivariant line bundle $L_{\lambda}$ on $Fl(r)$. Here $L_{i}=\V_{i}/\V_{i-1}$ and $\V_{i}$ are the tautological bundles of rank $i$ on $Fl(r)$.

\begin{lemma} \label{lemma 2} 
	$L_{1}^{\lambda_1}\otimes...\otimes L_{r}^{\lambda_r} \cong L_{\lambda}$ as equivariant line bundles on $Fl(r)$ with $\lambda=(\lambda_1,...,\lambda_r)$.
\end{lemma}

\begin{proof}
	We show that the two line bundles are isomorphic by showing that their fibres are isomorphic.
	
	Considering the standard flag $(0 \subset \CC<e_1> \subset \CC<e_1,e_2> \subset .... \subset \CC<e_1,..,e_r>)$, we have any flag $F$ is of the form
	\[
	F=(0 \subset \CC<ge_1> \subset \CC<ge_1,ge_2> \subset ... \subset \CC<ge_1,...,ge_r>)
	\] for some $g \in G$. The fibre of $L_{1}^{\lambda_1}\otimes...\otimes L_{r}^{\lambda_r}$ over $F$ is 
	\[
	(L_{1}^{\lambda_1}\otimes...\otimes L_{r}^{\lambda_r})_{F} = \CC<(ge_1)^{\lambda_1} \otimes ... \otimes (ge_{r})^{\lambda_r}>.
	\]
	
	On the other hand, the flag $F$ corresponds to the element $gB$ in $G/B$ under the isomorphism $Fl(r) \cong G/B$, we have the fibre of $L_{\lambda}$ over $gB$ is 
	\[
	(L_{\lambda})_{gB}=\{(g,t)B|t\in \CC \}.
	\]
	
	Defining the map $\phi:\{(g,t)B|t\in \CC \} \rightarrow \CC<(ge_1)^{\lambda_1} \otimes ... \otimes (ge_{r})^{\lambda_r}>$ by $\phi((g,t)B)=t (ge_1)^{\lambda_1} \otimes ... \otimes (ge_{r})^{\lambda_r}$ for all $t \in \CC$. Then clearly this is an isomorphism of vector spaces. We have to check that this map is well-defined.
	
	For any upper triangular matrix $b=(a_{ij}) \in B$, we have 
	\begin{align*}
		\phi( (gb^{-1},\lambda(b)t)B)=\phi((gb^{-1},a_{11}^{\lambda_1}...a_{rr}^{\lambda_r}t)B) &=a_{11}^{\lambda_1}...a_{rr}^{\lambda_r}t(gb^{-1}e_{1})^{\lambda_1}...(gb^{-1}e_{r})^{\lambda_r} \\
		&=a_{11}^{\lambda_1}...a_{rr}^{\lambda_r}ta_{11}^{-\lambda_1}...a_{rr}^{-\lambda_r}(ge_1)^{\lambda_1} \otimes ... \otimes (ge_{r})^{\lambda_r} \\
		&=t(ge_1)^{\lambda_1} \otimes ... \otimes (ge_{r})^{\lambda_r} \\
		&=\phi( (g,t)B)
	\end{align*}
	
	It is easy to see that $L_{1}^{\lambda_1}\otimes...\otimes L_{r}^{\lambda_r}$ is a $G$-equivariant line bundle. Thus we have $L_{1}^{\lambda_1}\otimes...\otimes L_{r}^{\lambda_r} \cong L_{\lambda}$.
\end{proof}

Thus if we apply the Borel-Weil-Bott theorem to each fibre of $\pi$, we would get $\bo_{(k,N-k)} {\Ff}_{\lambda_{r}} \ast {\Ff}_{\lambda_{r-1}} \ast ...\ast {\Ff}_{\lambda_{1}}$ as a family of irreducible representations of $G$ of certain highest weight. We have the following theorem, which is a direct generalization of Theorem \ref{Theorem 4}.

\begin{theorem} \label{Theorem 5}
	Let $\lambda=(\lambda_{1},...,\lambda_{r})$. Then we have the following isomorphism of FM kernels
	\begin{align*}
		&\bo_{(k,N-k)} {\Ff}_{\lambda_{r}} \ast {\Ff}_{\lambda_{r-1}} \ast ...\ast {\Ff}_{\lambda_{1}} \\
		&\cong \begin{cases}
			\iota_{2*}(\s_{-\lambda}\V_{k}/\V_{k-r})^{\vee} & if -\lambda\ is \ dominant \\
			0 & if -\lambda \  \text{is  not  dominant  and} \ \nexists \ w \in W \ s.t. \ w \star (-\lambda) \ is \ dominant \\
			\iota_{2*}(\s_{w \star (-\lambda)}\V_{k}/\V_{k-r})^{\vee}[-l(w)] & if -\lambda \ is \ not \ dominant \ and \ \exists!  \ w \in W \ s.t. \ w \star (-\lambda) \ is \ dominant   
		\end{cases}
	\end{align*}
\end{theorem}

\begin{remark}
	We can use the Grothendieck-Veridier duality to verify the above theorem. Define $\LL_{\lambda}:=(\V_{k-r+1}/\V_{k-r})^{\lambda_{1}} \otimes ... \otimes (\V_{k}/\V_{k-1})^{\lambda_{r}}$ for $\lambda=(\lambda_{1},...,\lambda_{r})$. Since we have to calculate $\pi_{*}\LL_{\lambda}$, by Grothendieck-Verdier duality, we have 
	\[
	\DD_{Fl(k-r,k)}\pi_{*}\LL_{\lambda} \cong \pi_{*}\DD_{Fl(k-r,k-r+1,...,k)}\LL_{\lambda}
	\] after simplification we obtain the following 
	\begin{equation} \label{eq 31}
		(\pi_{*}\LL_{\lambda})^{\vee} \cong \pi_{*}\LL_{2\rho-\lambda} \otimes \det(\V_{k}/\V_{k-r})^{-r+1}[l(w_{0})].
	\end{equation}
	
	If $-\lambda=(-\lambda_{1},...,-\lambda_{r})$ is dominant, then we have $-\lambda_{1} \geq -\lambda_{2} \geq ... \geq -\lambda_{r}$, so $\lambda_{1} \leq \lambda_{2} \leq .. \leq \lambda_{r}$. By Theorem \ref{Theorem 5}, we have $\pi_{*}\LL_{\lambda} \cong (\s_{-\lambda}\V_{k}/\V_{k-r})^{\vee}$ and thus $(\pi_{*}\LL_{\lambda})^{\vee} \cong \s_{-\lambda}\V_{k}/\V_{k-r}$.
	
	On the other hand, since $-\lambda$ is dominant, we have $2\rho-\lambda=(2r-2-\lambda_{1},2r-4-\lambda_{2},...,-\lambda_{r})$ is also dominant, which means that $-2\rho+\lambda$ is not dominant. Choose $w=w_{0}$ to be the longest element in the Weyl group, then 
	\begin{align*}
		w_{0} \star (-2\rho+\lambda) &= w_{0}(-\rho+\lambda)-\rho=w_{0}(\lambda_{1}-r+1,...,\lambda_{r})-\rho \\
		&= (\lambda_{r},...,\lambda_{1}-r+1)-\rho = (\lambda_{r}-r+1,...,\lambda_{1}-r+1)=w_{0}\lambda-r+1
	\end{align*} which is dominant.
	
	So from  Theorem \ref{Theorem 5}, we have
	\begin{align*}
		\pi_{*}\LL_{2\rho-\lambda}  &\cong (\s_{w \star (-2\rho+\lambda)}\V_{k}/\V_{k-r})^{\vee}[-l(w_{0})]  \cong \s_{-w_{0}(w \star (-2\rho+\lambda))}\V_{k}/\V_{k-r}[-l(w_{0})] \\
		& \cong \s_{-w_{0}(w_{0}\lambda-r+1)}\V_{k}/\V_{k-r}[-l(w_{0})]  \cong  \s_{-\lambda+r-1}\V_{k}/\V_{k-r}[-l(w_{0})] \\
		& \cong \s_{-\lambda}\V_{k}/\V_{k-r} \otimes \det(\V_{k}/\V_{k-r})^{r-1}[-l(w_{0})].
	\end{align*}
	
	This implies that
	\[
	\pi_{*}\LL_{2\rho-\lambda} \otimes \det(\V_{k}/\V_{k-r})^{-r+1}[l(w_{0})] \cong \s_{-\lambda}\V_{k}/\V_{k-r} \cong (\pi_{*}\LL_{\lambda})^{\vee} 
	\] which verify the theorem. For other cases are similar.
\end{remark}

Observe that $\GG(k,N)=\GG(0,N) \times \GG(k,N) $. We can view $\s_{\lambda}\V$ as an object in $\Dd^b(\GG(0,N) \times \GG(k,N))$. So we can use Theorem \ref{Theorem 5} to write $\s_{\lambda}\V$ in terms of convolutions of the kernels $\Ff_{\lambda_{i}}$. More precisely, we have the following corollary.

\begin{corollary} \label{Corollary 6}
	Let $\lambda=(\lambda_{1},...,\lambda_{k})$ with $0 \leq \lambda_{k} \leq ... \leq \lambda_{1} \leq N-k$ and  $\mu=(\mu_{1},...,\mu_{N-k})$ with $0 \leq \mu_{N-k} \leq  ... \leq \mu_{2} \leq \mu_{1} \leq k$.. Then 
	\begin{align*}
	& \s_{\lambda}\V 
		\cong \Ff_{\lambda_{1}} \ast ... \ast \Ff_{\lambda_{k}} \bo_{(0,N)} \in \Dd^b(\GG(0,N) \times \GG(k,N))  \\
	& \s_{\mu}(\CC^N/\V)^{\vee}
		\cong \Ee_{-\mu_{1}} \ast ... \ast \Ee_{-\mu_{N-k}} \bo_{(N,0)} \in \Dd^b(\GG(N,N) \times \GG(k,N))
	\end{align*}
\end{corollary}

\subsection{The main theorem}

We state the main theorem, which is similar to Theorem 4.1 in \cite{CKo}.

\begin{theorem} \label{Theorem 7}
Given a categorical action of $\dot{\Uu}_{0,N}(L\SL_{2})$ on $\bigoplus_{k} \Kk(k,N-k)$. Suppose that for each $k \geq 1$, the image of the functor 
\begin{equation*}
\varphi_{k}:=  \bigoplus_{\lambda \in P(N-k,k)} {\F}_{\lambda} {\bo}_{(0,N)} :\Kk(0,N) \rightarrow \Kk(k,N-k)
\end{equation*}
weakly generate $\Kk(k,N-k)$, where ${\F}_{\lambda} {\bo}_{(0,N)}:= {\F}_{\lambda_{1}}{\F}_{\lambda_{2}}...{\F}_{\lambda_{k}}{\bo}_{(0,N)}$ for $\lambda=(\lambda_{1},...,\lambda_{k}) \in P(N-k,k)$. Moreover, Suppose that $\Kk(0,N)$ is equipped with a t-structure $(\Kk(0,N)^{\leq 0},\Kk(0,N)^{\geq 0})$ such that $({\F}_{\lambda})^{R}{\F}_{\lambda'}{\bo}_{(0,N)}$ is t-exact for all $\lambda, \lambda' \in P(N-k,k)$.  

Then one can uniquely extend this t-structure to all other categories $\Kk(k,N-k)$ so that $({\F}_{N-k-1}\bo_{(k,N-k)})^{R}$ and $({\E}\bo_{(k,N-k)})^{R}({\spi}^{-})[k-2]$ are t-exact.
\end{theorem}

\begin{proof} 
Consider the right adjoint functor 

\begin{equation}\label{eq 32} 
(\varphi_{k})^{R}:=  \bigoplus_{\lambda \in P(N-k,k)} ({\F}_{\lambda} {\bo}_{(0,N)})^{R} :\Kk(k,N-k) \rightarrow \Kk(0,N).
\end{equation}

Then we have 
\begin{equation*}
(\varphi_{k})^{R} \circ (\varphi_{k}) = \bigoplus_{\lambda, \lambda' \in P(N-k,k)} ({\F}_{\lambda})^{R}{\F}_{\lambda'}{\bo}_{(0,N)}:\Kk(0,N) \rightarrow \Kk(0,N).
\end{equation*}

By the assumption, $({\F}_{\lambda})^{R}{\F}_{\lambda'}{\bo}_{(0,N)}$ is t-exact for all $\lambda, \lambda' \in P(N-k,k)$, so we get that $(\varphi_{k})^{R} \circ (\varphi_{k})$ is also t-exact. Moreover, since the image of $\varphi_{k}$ weakly generate $\Kk(k,N-k)$, we have $(\varphi_{k})^{R}$ is conservative.

Applying Theorem \ref{Theorem 8} and Lemma \ref{lemma 8} to the functor $(\varphi_{k})^{R}$, the induced t-structure on $\Kk(k,N-k)$ is given by
\begin{align} \label{eq 42}
	\begin{split}
	& \Kk(k,N-k)^{\leq 0} =\{ X \in \Kk(k,N-k) \ | \ (\varphi_{k})^{R}(X) \in \Kk(0,N)^{\leq 0} \}, \\
	& \Kk(k,N-k)^{\geq 0} =\{ X \in \Kk(k,N-k) \ | \ (\varphi_{k})^{R}(X) \in \Kk(0,N)^{\geq 0} \}.
		\end{split}
\end{align}

Next, we check that the functors $({\F}_{N-k-1}\bo_{(k,N-k)})^{R}$ and $({\E}\bo_{(k,N-k)})^{R}({\spi}^{-})[k-2]$ are t-exact on those induced t-structures. To check this, we have to show that given $X \in \Kk(k,N-k)^{\leq 0}$ and $Y \in \Kk(k-1,N-k+1)^{\leq 0}$, we have $({\F}_{N-k}\bo_{(k-1,N-k+1)})^{R}(X) \in  \Kk(k-1,N-k+1)^{\leq 0}$ and $({\E}\bo_{(k,N-k)})^{R}({\spi}^{-})[k-2](Y) \in  \Kk(k,N-k)^{\leq 0}$. 

To show that $({\F}_{N-k}\bo_{(k-1,N-k+1)})^{R}(X) \in  \Kk(k-1,N-k+1)^{\leq 0}$, we need to calculate $(\varphi_{k-1})^{R} \circ ({\F}_{N-k}\bo_{(k-1,N-k+1)})^{R}$. The definition of $(\varphi_{k-1})^{R} $ is similar to (\ref{eq 32}) with each direct summand $({\F}_{\lambda} {\bo}_{(0,N)})^{R}$,
where $\lambda=(\lambda_{1},..,\lambda_{k-1}) \in P(N-k+1,k-1)$. So we calculate each term of $(\varphi_{k-1})^{R} \circ ({\F}_{N-k}\bo_{(k-1,N-k+1)})^{R}$
\begin{equation*}
({\F}_{\lambda} {\bo}_{(0,N)})^{R}({\F}_{N-k}{\bo}_{(k-1,N-k+1)})^{R} \cong  ({\F}_{N-k}{\F}_{\lambda_{1}}...{\F}_{\lambda_{k-1}} {\bo}_{(0,N)})^{R}
\end{equation*}

Note that $\lambda=(\lambda_{1},..,\lambda_{k-1}) \in P(N-k+1,k-1)$, so $ 0 \leq \lambda_{k-1} \leq ... \leq \lambda_{1} \leq N-k+1$. If $\lambda_{i}=N-k+1$ for some $1 \leq i \leq k-1$, then we have $\lambda_{j}=N-k+1$ for all $j \leq i$ and thus ${\F}_{N-k}{\F}_{\lambda_{1}}...{\F}_{\lambda_{k-1}} {\bo}_{(0,N)} \cong 0$ by relation (8) in Definition \ref{definition 2}. Otherwise, we have $ 0 \leq \lambda_{k-1} \leq ... \leq \lambda_{1} \leq N-k$ and thus $({\F}_{N-k}{\F}_{\lambda_{1}}...{\F}_{\lambda_{k-1}} {\bo}_{(0,N)})^{R} $ is a direct summand of $(\varphi_{k})^{R}$.

The above argument implies that $(\varphi_{k-1})^{R} \circ ({\F}_{N-k}\bo_{(k-1,N-k+1)})^{R}$ can be written as a direct sum whose each summand is a direct summand of  $(\varphi_{k})^{R}$. Now, since $X \in \Kk(k,N-k)^{\leq 0}$, we have $(\varphi_{k})^{R}(X) \in \Kk(0,N)^{\leq 0}$. By Lemma \ref{lemma 9}, we conclude that $(\varphi_{k-1})^{R} \circ ({\F}_{N-k}\bo_{(k-1,N-k+1)})^{R}(X) \in \Kk(k,N-k)^{\leq 0}$, which implies $({\F}_{N-k}\bo_{(k-1,N-k+1)})^{R}(X)  \in \Kk(k-1,N-k+1)^{\leq 0}$.

Next, we show that  $({\E}\bo_{(k,N-k)})^{R}({\spi}^{-})[k-2](Y) \in  \Kk(k,N-k)^{\leq 0}$. Similarly as the case for showing $({\F}_{N-k}\bo_{(k-1,N-k+1)})^{R}(X) \in  \Kk(k-1,N-k+1)^{\leq 0}$, we have to calculate $(\varphi_{k})^{R} \circ ({\E}\bo_{(k,N-k)})^{R}({\spi}^{-})[k-2]$. By (\ref{eq 32}), we calculate each term 
\begin{equation}  \label{eq 41}
	({\F}_{\lambda} {\bo}_{(0,N)})^{R} ({\E}\bo_{(k,N-k)})^{R}({\spi}^{-})[k-2] \cong ({\E}{\F}_{\lambda_{1}}...{\F}_{\lambda_{k}} {\bo}_{(0,N)})^{R}  ({\spi}^{-})[k-2]
\end{equation} where $\lambda=(\lambda_{1},...,\lambda_{k}) \in P(N-k,k)$.

Now $\lambda=(\lambda_{1},..,\lambda_{k}) \in P(N-k,k)$, so $ 0 \leq \lambda_{k} \leq ... \leq \lambda_{1} \leq N-k$. By relation (13)(c)(e) in Definition \ref{definition 2}, we obtain that 
\begin{equation*}
	{\E}{\F}_{\lambda_{1}}...{\F}_{\lambda_{k}} {\bo}_{(0,N)} \cong  ... \cong {\F}_{\lambda_{1}}...{\E}{\F}_{\lambda_{k}} {\bo}_{(0,N)} \cong \begin{cases}
		0  & \text{if} \ 1\leq \lambda_{k} \leq N-k \\
		{\F}_{\lambda_{1}}...{\F}_{\lambda_{k-1}} {\spi}^{-}[-1] {\bo}_{(0,N)}   & \text{if}  \ \lambda_{k}=0
	\end{cases}
\end{equation*}

Hence, in the case where ${\E}{\F}_{\lambda_{1}}...{\F}_{\lambda_{k}} {\bo}_{(0,N)} \ncong 0$, (\ref{eq 41}) becomes
\begin{align*}
&({\E}{\F}_{\lambda_{1}}...{\F}_{\lambda_{k}} {\bo}_{(0,N)})^{R} ({\spi}^{-})[k-2] \\
&\cong ({\F}_{\lambda_{1}}...{\F}_{\lambda_{k-1}} {\spi}^{-}[-1] {\bo}_{(0,N)})^{R}({\spi}^{-})[k-2] \cong ({\F}_{\lambda_{1}}...{\F}_{\lambda_{k-1}} {\spi}^{-}[-1] {\bo}_{(0,N)})^{R}(({\spi}^{-})^{-1})^{R} [k-2] \\
& \cong (({\spi}^{-})^{-1}{\F}_{\lambda_{1}}...{\F}_{\lambda_{k-1}} {\spi}^{-}[-1] {\bo}_{(0,N)})^{R}[k-2] \cong  (({\F}_{\lambda_{1}+1}...{\F}_{\lambda_{k-1}+1} [k-2] {\bo}_{(0,N)})^{R}[k-2] \\
& \cong (({\F}_{\lambda_{1}+1}...{\F}_{\lambda_{k-1}+1}{\bo}_{(0,N)})^{R}
\end{align*} and $ 1 \leq \lambda_{k-1}+1 \leq ... \leq \lambda_{1}+1 \leq N-k+1$, which means that it is a direct summand of $(\varphi_{k-1})^{R}$.

The above argument implies that $(\varphi_{k})^{R} \circ ({\E}\bo_{(k,N-k)})^{R}({\spi}^{-})[k-2]$ can be written as a direct sum whose each summand is a direct summand of  $(\varphi_{k-1})^{R}$. By the same argument as above and using Lemma \ref{lemma 9}, we prove that $({\E}\bo_{(k,N-k)})^{R}({\spi}^{-})[k-2](Y) \in  \Kk(k,N-k)^{\leq 0}$. 

Finally, all the above argument hold for the cases $X \in \Kk(k,N-k)^{\geq 0}$ and $Y \in \Kk(k-1,N-k+1)^{\geq 0}$.
\end{proof}

The next corollary says that we can modify the functors $\varphi_{k}$ to get a different t-structure  so that $({\F}_{N-k-1}\bo_{(k,N-k)})^{L}$ and $({\E}\bo_{(k,N-k)})^{L}({\spi}^{-})[k-2]$ are t-exact.

\begin{corollary} \label{Corollary 8} 
Let  $\bigoplus_{k} \Kk(k,N-k)$ and $\varphi_{k}$ be the categories and functors that satisfy the assumptions in Theorem \ref{Theorem 7}. Consider the functors
\begin{equation*}
\varphi'_{k}:= ({\spi^{-}})^{N-k} ({\spi^{+}})^{k} \bigoplus_{\lambda \in P(N-k,k)} {\F}_{\lambda} {\bo}_{(0,N)} [k(N-k)-N]:\Kk(0,N) \rightarrow \Kk(k,N-k).
\end{equation*}

Then one can use $\varphi'_{k}$ to induce t-structure to all other categories $\Kk(k,N-k)$ as in (\ref{eq 42}) so that $({\F}_{N-k-1}\bo_{(k,N-k)})^{L}$ and $({\E}\bo_{(k,N-k)})^{L}({\spi}^{-})[k-2]$ are t-exact.
\begin{proof}
By definition, we have $\varphi'_{k}=({\spi^{-}})^{N-k} ({\spi^{+}})^{k}\varphi_{k}[k(N-k)-N]$. From the assumption of Theorem \ref{Theorem 7}, we have the image of $\varphi_{k}$ weakly generate $\Kk(k,N-k)$. Since ${\spi^{-}}$, ${\spi^{+}}$ are invertible functors, and the homological shift $[1]$ is also invertible, it is easy to see that the image of $\varphi'_{k}$ weakly generate $\Kk(k,N-k)$.

Consider the right adjoint functor
\begin{equation} 
(\varphi'_{k})^{R}:=  \bigoplus_{\lambda \in P(N-k,k)} ({\F}_{\lambda} {\bo}_{(0,N)})^{R} ({\spi^{-}})^{-N+k} ({\spi^{+}})^{-k} [-k(N-k)+N]:\Kk(k,N-k) \rightarrow \Kk(0,N)
\end{equation} which is conservative from the above argument.

Moreover, we have
\begin{equation*}
	(\varphi'_{k})^{R} \circ (\varphi'_{k}) = \bigoplus_{\lambda, \lambda' \in P(N-k,k)} ({\F}_{\lambda})^{R}{\F}_{\lambda'}{\bo}_{(0,N)}=(\varphi_{k})^{R} \circ (\varphi_{k}):\Kk(0,N) \rightarrow \Kk(0,N).
\end{equation*}	which is also t-exact on the t-structure $(\Kk(0,N)^{\leq 0},\Kk(0,N)^{\geq 0})$ from the assumption.

Thus we can apply Theorem \ref{Theorem 8} and Lemma \ref{lemma 8} to the functor $(\varphi'_{k})^{R}$, the induced t-structure on $\Kk(k,N-k)$ is given by
\begin{align} \label{eq 43}
	\begin{split}
		& \Kk(k,N-k)^{' \leq 0} =\{ X \in \Kk(k,N-k) \ | \ (\varphi'_{k})^{R}(X) \in \Kk(0,N)^{\leq 0} \}, \\
		& \Kk(k,N-k)^{' \geq 0} =\{ X \in \Kk(k,N-k) \ | \ (\varphi'_{k})^{R}(X) \in \Kk(0,N)^{\geq 0} \}.
	\end{split}
\end{align}
	
Next, we check that the functors $({\F}_{N-k-1}\bo_{(k,N-k)})^{L}$ and $({\E}\bo_{(k,N-k)})^{L}({\spi}^{-})[k-2]$ are t-exact on those induced t-structures. Similarly as the proof of Theorem \ref{Theorem 7}, we have to calculate $(\varphi'_{k-1})^{R} \circ ({\F}_{N-k}\bo_{(k-1,N-k+1)})^{L}$ and $(\varphi'_{k})^{R} \circ ({\E}\bo_{(k,N-k)})^{L}({\spi}^{-})[k-2]$.

We calculate each term of $(\varphi'_{k-1})^{R} \circ ({\F}_{N-k}\bo_{(k-1,N-k+1)})^{L}$, which is the following
\begin{equation} \label{eq 33}
	({\F}_{\lambda} {\bo}_{(0,N)})^{R} ({\spi^{-}})^{-N+k-1} ({\spi^{+}})^{-k+1} \circ ({\F}_{N-k}\bo_{(k-1,N-k+1)})^{L} [-(k-1)(N-k+1)+N].
\end{equation} where $\lambda=(\lambda_{1},..,\lambda_{k-1}) \in P(N-k+1,k-1)$,

Using the definition of categorical action, we have 
\begin{align*}
	& ({\spi^{-}})^{-N+k-1} ({\spi^{+}})^{-k+1} \circ ({\F}_{N-k}\bo_{(k-1,N-k+1)})^{L} [-(k-1)(N-k+1)+N]  \\
	& \cong ({\spi^{-}})^{-N+k-1} ({\spi^{+}})^{-k+1}  {\E} ({\spi}^{+})^{-1} [1-(k-1)(N-k+1)+N] \bo_{(k,N-k)}\\
	& \cong ({\spi^{-}})^{-N+k-1}  {\E}_{-k+1} ({\spi}^{+})^{-k} [k-(k-1)(N-k+1)+N] \bo_{(k,N-k)}\\
	& \cong ({\spi^{-}})^{-N+k+1}  {\E}_{-k-1} ({\spi^{-}})^{-2} ({\spi}^{+})^{-k} [-2+k-(k-1)(N-k+1)+N] \bo_{(k,N-k)}\\
	& \cong  ({\F}_{N-k}{\bo}_{(k-1,N-k+1)})^{R} ({\spi^{-}})^{-N+k} ({\spi}^{+})^{-k} [(-N+k+1)-2+k-(k-1)(N-k+1)+N] \bo_{(k,N-k)} \\
	& \cong  ({\F}_{N-k}{\bo}_{(k-1,N-k+1)})^{R} ({\spi^{-}})^{-N+k} ({\spi}^{+})^{-k} [-k(N-k)+N] \bo_{(k,N-k)}.
\end{align*}

This implies that 
\begin{align*}
&(\varphi'_{k-1})^{R}({\F}_{N-k}\bo_{(k-1,N-k+1)})^{L} \\
&\cong (\varphi_{k-1})^{R} ({\spi^{-}})^{-N+k-1} ({\spi^{+}})^{-k+1}  ({\F}_{N-k}\bo_{(k-1,N-k+1)})^{L} [-(k-1)(N-k+1)+N]  \\
&\cong (\varphi_{k-1})^{R} ({\F}_{N-k}{\bo}_{(k-1,N-k+1)})^{R} ({\spi^{-}})^{-N+k} ({\spi}^{+})^{-k} [-k(N-k)+N] \bo_{(k,N-k)}.
\end{align*}

Thus (\ref{eq 33}) becomes 
\begin{align*}
	&({\F}_{\lambda} {\bo}_{(0,N)})^{R}({\F}_{N-k}{\bo}_{(k-1,N-k+1)})^{R} ({\spi^{-}})^{-N+k} ({\spi}^{+})^{-k} [-k(N-k)+N] \bo_{(k,N-k)} \\
	& \cong ({\F}_{N-k}{\F}_{\lambda_{1}}...{\F}_{\lambda_{k-1}} {\bo}_{(0,N)})^{R} ({\spi^{-}})^{-N+k} ({\spi}^{+})^{-k} [-k(N-k)+N] \bo_{(k,N-k)} .
\end{align*}

So we have $(\varphi'_{k-1})^{R} \circ ({\F}_{N-k}\bo_{(k-1,N-k+1)})^{L}$ can be written as a direct sum whose each summand is a direct summand of  $(\varphi'_{k})^{R}$. By the same argument as in the proof of Theorem \ref{Theorem 7}, we conclude that $({\F}_{N-k}\bo_{(k-1,N-k+1)})^{L}$ is t-exact.

Next, we check the t-exactness for $({\E}\bo_{(k,N-k)})^{L}({\spi}^{-})[k-2]$. Again, we calculate each term of $(\varphi'_{k})^{R} \circ ({\E}\bo_{(k,N-k)})^{L}({\spi}^{-})[k-2]$, which is the following
\begin{equation} \label{eq 34}
	({\F}_{\lambda} {\bo}_{(0,N)})^{R} ({\spi^{-}})^{-N+k} ({\spi^{+}})^{-k} \circ ({\E}\bo_{(k,N-k)})^{L}({\spi}^{-})[k-2][-k(N-k)+N]
\end{equation} where $\lambda=(\lambda_{1},...,\lambda_{k}) \in P(N-k,k)$.

Using the definition of categorical action, we have 
\begin{align*}
	&({\spi^{-}})^{-N+k} ({\spi^{+}})^{-k} \circ ({\E}\bo_{(k,N-k)})^{L} [-k(N-k)+N] \\
	& \cong ({\spi^{-}})^{-N+k} ({\spi^{+}})^{-k} ({\spi}^{-})^{k-1} {\F}  ({\spi}^{-})^{-k}{\bo}_{(k-1,N-k+1)} [k-k(N-k)+N] \\
	&\cong  ({\spi^{+}})^{-k}  {\F}_{N-2k+1}  ({\spi}^{-})^{-N+k-1} {\bo}_{(k-1,N-k+1)} [k-k(N-k)+N+N-2k+1] \\
	& \cong  ({\spi^{+}})  {\F}_{N-k+2}  ({\spi}^{-})^{-N+k-1} ({\spi^{+}})^{-k-1} {\bo}_{(k-1,N-k+1)} [-k(N-k)+N+N-2k] \\
	& \cong  ({\E}{\bo_{(k,N-k)}})^{R}  ({\spi}^{-})^{-N+k-1} ({\spi^{+}})^{-k+1} {\bo}_{(k-1,N-k+1)} [-k(N-k)+N+N-2k+1] \\
	& \cong  ({\E}{\bo_{(k,N-k)}})^{R}  ({\spi}^{-})^{-N+k-1} ({\spi^{+}})^{-k+1} {\bo}_{(k-1,N-k+1)} [-(k-1)(N-k+1)+N] 
\end{align*}

This implies that 
\begin{align*}
	&(\varphi'_{k})^{R} \circ ({\E}\bo_{(k,N-k)})^{L}({\spi}^{-})[k-2]\\
	&\cong (\varphi_{k})^{R} ({\spi^{-}})^{-N+k} ({\spi^{+}})^{-k}  ({\E}\bo_{(k,N-k)})^{L}({\spi}^{-})[k-2][-k(N-k)+N]  \\
	&\cong (\varphi_{k})^{R} ({\E}{\bo_{(k,N-k)}})^{R}({\spi}^{-})  ({\spi}^{-})^{-N+k-1} ({\spi^{+}})^{-k+1} {\bo}_{(k-1,N-k+1)} [k-2][-(k-1)(N-k+1)+N] 
\end{align*}

Thus (\ref{eq 34}) becomes 
\begin{align} \label{eq 35}
	\begin{split}
		&({\F}_{\lambda} {\bo}_{(0,N)})^{R} ({\E}{\bo_{(k,N-k)}})^{R}  ({\spi}^{-})^{-N+k-1} ({\spi^{+}})^{-k+1}({\spi}^{-})[k-2] [-(k-1)(N-k+1)+N] {\bo}_{(k-1,N-k+1)} \\
		& \cong ({\E}{\F}_{\lambda_{1}}...{\F}_{\lambda_{k}} {\bo}_{(0,N)})^{R}  ({\spi}^{-})^{-N+k-1} ({\spi^{+}})^{-k+1}({\spi}^{-})[k-2] [-(k-1)(N-k+1)+N] {\bo}_{(k-1,N-k+1)}
	\end{split} 
\end{align}

By the proof of Theorem \ref{Theorem 7}, in the case where ${\E}{\F}_{\lambda_{1}}...{\F}_{\lambda_{k}} {\bo}_{(0,N)} \ncong 0$, (\ref{eq 35}) becomes 
\begin{align*}
	&({\E}{\F}_{\lambda_{1}}...{\F}_{\lambda_{k}} {\bo}_{(0,N)})^{R}  ({\spi}^{-})^{-N+k-1} ({\spi^{+}})^{-k+1}({\spi}^{-})[k-2] [-(k-1)(N-k+1)+N] {\bo}_{(k-1,N-k+1)} \\ 
	&\cong (({\F}_{\lambda_{1}+1}...{\F}_{\lambda_{k-1}+1}{\bo}_{(0,N)})^{R}({\spi}^{-})^{-N+k-1} ({\spi^{+}})^{-k+1} [-(k-1)(N-k+1)+N] {\bo}_{(k-1,N-k+1)}
\end{align*} and $ 1 \leq \lambda_{k-1}+1 \leq ... \leq \lambda_{1}+1 \leq N-k+1$, which means that it is a direct summand of $(\varphi'_{k-1})^{R}$. Hence, we show that $({\E}\bo_{(k,N-k)})^{L}({\spi}^{-})[k-2]$ is t-exact.
\end{proof}
\end{corollary}

\subsection{Application  to $\Dd^b(\GG(k,N))$}

From Theorem \ref{Theorem 3}, we have a categorical action of $\dot{\Uu}_{0,N}(L\SL_{2})$ on $\bigoplus_{k} \Dd^b(\GG(k,N))$.  We would like to apply Corollary \ref{Corollary 8} to get t-structures on $\Dd^b(\GG(k,N))$.

For the highest weight $(0,N)$, we have $\GG(0,N)=pt$. So the highest weight category is $\Dd^b(\GG(0,N)) = \Dd^b(\text{Vect})$, which is the category of finite dimensional graded vector spaces. We equip  $\Dd^b(\GG(0,N))$ with the standard t-structure $(\Dd^{\leq 0}(0,N), \Dd^{\geq 0}(0,N))$; i.e., 
\begin{align*}
& \Dd^{\leq 0}(0,N)=\langle \{ \CC[d] \ | \ d \geq 0 \} \rangle \\
& \Dd^{\geq 0}(0,N)=\langle \{ \CC[d] \ | \ d \leq 0 \} \rangle.
\end{align*}

We would like to apply Corollary \ref{Corollary 8} to obtain t-structures $(\Dd_{act}^{\leq 0}(k,N-k), \Dd_{act}^{\geq 0}(k,N-k))$ on each weight category $\Dd^b(\GG(k,N))$. So we need to know what are the functors $\varphi'_{k}$ in Corollary \ref{Corollary 8} when the categories are $\Dd^b(\GG(k,N))$.

Since we are in the geometric setting, all the functors are given by FM kernels and composition of functors are given by convolution of kernels. Recall that the definition of $\varphi'_{k}$.
\begin{equation*}
\varphi'_{k}:= ({\spi^{-}})^{N-k} ({\spi^{+}})^{k} \bigoplus_{\lambda \in P(N-k,k)} {\F}_{\lambda} {\bo}_{(0,N)} [k(N-k)-N]:\Dd^b(\GG(0,N)) \rightarrow \Dd^b(\GG(k,N))
\end{equation*}

By Definition \ref{definition 3}, we know what are the kernels for ${\spi^{\pm}}$ and ${\F}_{\lambda_{i}}$. An easy calculation tells us that the kernel for $({\spi^{-}})^{N-k} ({\spi^{+}})^{k}$ is given by 
\begin{equation*}
({\Psi^{-}})^{N-k} \ast  ({\Psi^{+}})^{k} \cong \Delta(k)_{*}\tdet (\V)^{-N+k} \otimes \tdet (\CC^N/\V)^{k} [N-2k(N-k)] \in \Dd^b(\GG(k,N) \times \GG(k,N))
\end{equation*}

Note that $\tdet (\V)^{-N+k} \otimes \tdet (\CC^N/\V)^{k}$ is the anti-canonical bundle, denoted by $\omega^{-1}_{\GG(k,N)}$ on $\GG(k,N)$. On the other hand, by Corollary \ref{Corollary 6}, we know that the kernel for ${\F}_{\lambda} {\bo}_{(0,N)}={\F}_{\lambda_{1}}... {\F}_{\lambda_{k}}{\bo}_{(0,N)}$ is given by 
\begin{equation*}
\Ff_{\lambda_{1}} \ast ... \ast \Ff_{\lambda_{k}} \bo_{(0,N)} \cong \s_{\lambda}\V  \in \Dd^b(\GG(0,N) \times \GG(k,N)).
\end{equation*}

Thus the kernel for $\varphi_{k}$ is given by 
\begin{equation} \label{eq 36}
\omega^{-1}_{\GG(k,N)} \otimes \bigoplus_{\lambda \in P(N-k,k)} \s_{\lambda}\V [-k(N-k)]  \in \Dd^b(\GG(0,N) \times \GG(k,N)).
\end{equation}

To apply Corollary \ref{Corollary 8}, first, we have to check that the image of $\varphi'_{k}$ weakly generate $\Dd^b(\GG(k,N))$. This is easy to see since that $\bigoplus_{\lambda \in P(N-k,k)} \s_{\lambda}\V$ is a tilting bundle on $\Dd^b(\GG(k,N))$, and tensoring by an line bundle or homological shifts do not change the tilting property. We have $\GG(0,N)=pt$ and $\Dd^b(\GG(0,N))$ is generated by the one-dimensional vector space $\CC$. So the image of $\varphi'_{k}$ is a tilting bundle, which weakly generate $\Dd^b(\GG(k,N))$.

Next, we have to check $(\varphi'_{k})^{R} \varphi'_{k}$ is t-exact on the standard t-structure  $(\Dd^{\leq 0}(0,N), \Dd^{\geq 0}(0,N))$. We already know the kernel for $\varphi'_{k}$ is (\ref{eq 36}), by Proposition \ref{Proposition 1}, we have the kernel for right adjoint functor $(\varphi'_{k})^{R}$ is given by 
\begin{equation} \label{eq 38}
	\omega_{\GG(k,N)} \otimes \bigoplus_{\lambda \in P(N-k,k)} (\s_{\lambda}\V)^{\vee} [k(N-k)]  \in \Dd^b(\GG(k,N) \times \GG(0,N)).
\end{equation}

Thus $(\varphi'_{k})^{R} \varphi'_{k}$ is given by convolution of kernels 
\begin{align} \label{eq 37}
	\begin{split}
		& \{\omega_{\GG(k,N)} \otimes \bigoplus_{\lambda \in P(N-k,k)} (\s_{\lambda}\V)^{\vee} [k(N-k)]\} \ast \{\omega^{-1}_{\GG(k,N)} \otimes \bigoplus_{\lambda \in P(N-k,k)} \s_{\lambda}\V [-k(N-k)] \} \\
		& \cong \bigoplus_{\lambda, \lambda' \in P(N-k,k)} \pi_{13*}((\s_{\lambda}\V)^{\vee} \otimes \s_{\lambda'}\V)
	\end{split}
\end{align} where $\pi_{13}:\GG(0,N) \times \GG(k,N) \times \GG(0,N) \rightarrow \GG(0,N) \times \GG(0,N)$ is the natural projection.

Since $\pi_{13}$ is just the projection from $\GG(k,N)$ to a point, $\pi_{13*}((\s_{\lambda}\V)^{\vee} \otimes \s_{\lambda'}\V)$ is precisely the sheaf cohomology of $(\s_{\lambda}\V)^{\vee} \otimes \s_{\lambda'}\V$; i.e.,  $\pi_{13*}((\s_{\lambda}\V)^{\vee} \otimes \s_{\lambda'}\V) = H^{*}(\GG(k,N),(\s_{\lambda}\V)^{\vee} \otimes \s_{\lambda'}\V)$.

We know that 
\begin{equation*}
H^{*}(\GG(k,N),(\s_{\lambda}\V)^{\vee} \otimes \s_{\lambda'}\V) \cong R\Hom(\Oo_{\GG(k,N)}, (\s_{\lambda}\V)^{\vee} \otimes \s_{\lambda'}\V) \cong \Ext^{*}(\s_{\lambda}\V, \s_{\lambda'}\V)
\end{equation*} and from Corollary \ref{Corollary 2}, $\R_{(k,N-k)}$ is a strong full exceptional collection. This implies that $\Ext^{*}(\s_{\lambda}\V, \s_{\lambda'}\V)$ only have degree 0 part; i.e., $\Hom(\s_{\lambda}\V, \s_{\lambda'}\V)$. 

So (\ref{eq 37}) becomes $\bigoplus_{\lambda, \lambda' \in P(N-k,k)}\Hom(\s_{\lambda}\V, \s_{\lambda'}\V)$, which is a vector space concentrate in homological degree 0. Thus $(\varphi'_{k})^{R} \varphi'_{k}$ is t-exact on the standard t-structure.

Now we can apply Corollary \ref{Corollary 8} to get the induced t-structure, denoted by $(\Dd_{act}^{\leq 0}(k,N-k), \Dd_{act}^{\geq 0}(k,N-k))$ on $\Dd^b(\GG(k,N))$. It has the following description
\begin{align*}
	\Dd_{act}^{\geq 0}(k,N-k) & := \{X \in \Dd^b(\GG(k,N)) \ | \ (\varphi'_{k})^{R}(X) \in \Dd^{\geq 0}(0,N) \}, \\
	\Dd_{act}^{\leq 0}(k,N-k)  & := \{X \in \Dd^b(\GG(k,N)) \ | \ (\varphi'_{k})^{R}(X) \in \Dd^{\leq 0}(0,N) \}. 
\end{align*}

The main result of this section is the following.

\begin{theorem} \label{Theorem 9}
The induced t-structure on $\Dd^b(\GG(k,N))$ from the standard t-structure on $\Dd^b(\GG(0,N))$ under the categorical $\dot{\Uu}_{0,N}(L\SL_{2})$ action is precisely the t-structure from Corollary \ref{Corollary 4}. More precisely, we have
\begin{align*}
&\Dd_{act}^{\geq 0}(k,N-k) := \{X \in \Dd^b(\GG(k,N)) \ | \ (\varphi'_{k})^{R}(X) \in \Dd^{\geq 0}(0,N) \} \\
& = \Dd_{ex}^{\geq 0}(k,N-k)=\langle \{\s_{\lambda}\V[d] \ | \ \lambda \in P(N-k,k), \ d \leq 0 \} \rangle, \\
&\Dd_{act}^{\leq 0}(k,N-k)  := \{X \in \Dd^b(\GG(k,N)) \ | \ (\varphi'_{k})^{R}(X) \in \Dd^{\leq 0}(0,N) \} \\
&= \Dd_{ex}^{\leq 0}(k,N-k)=\langle \{\s_{\mu}(\CC^N/\V)[-|\mu|+d]  \ | \ \mu \in P(k,N-k), \ d \geq 0 \} \rangle.
\end{align*}
\end{theorem}

\begin{proof}
To prove this theorem, it suffices to prove $\Dd_{ex}^{\leq 0}(k,N-k) \subset \Dd_{act}^{\leq 0}(k,N-k)$ and  $\Dd_{ex}^{\geq 0}(k,N-k) \subset \Dd_{act}^{\geq 0}(k,N-k)$. Then by Lemma \ref{lemma 10} to get the result. 

From Corollary \ref{Corollary 4}, we know that the t-structure $(\Dd_{ex}^{\leq 0}(k,N-k), \Dd_{ex}^{\geq 0}(k,N-k))$ can be described by exceptional collections. So to prove $\Dd_{ex}^{\leq 0}(k,N-k) \subset \Dd_{act}^{\leq 0}(k,N-k)$ and  $\Dd_{ex}^{\geq 0}(k,N-k) \subset \Dd_{act}^{\geq 0}(k,N-k)$, it suffice to check on each generators; i.e., $\s_{\lambda}\V[d] \in \Dd_{act}^{\geq 0}(k,N-k)$ for all $d \leq 0$ and $\s_{\mu}(\CC^N/\V)[-|\mu|+d'] \in  \Dd_{act}^{\leq 0}(k,N-k)$ for all $d' \geq 0$.

First, we show that  $\s_{\mu}(\CC^N/\V)[-|\mu|+d'] \in  \Dd_{act}^{\leq 0}(k,N-k)$ for all $d' \geq 0$. To show this, it suffices to show that $(\varphi'_{k})^{R}(\s_{\mu}(\CC^N/\V)[-|\mu|+d']) \in  \Dd^{\leq 0}(0,N)$. Since $(\varphi'_{k})^{R}$ is a FM transform with FM kernel given by (\ref{eq 38}), by definition we have to calculate the following 
\begin{align} \label{eq 39}
	\begin{split}
	&(\varphi'_{k})^{R}(\s_{\mu}(\CC^N/\V)[-|\mu|+d'])=\pi_{2*}(\pi^{*}_{1}(\s_{\mu}(\CC^N/\V)[-|\mu|+d']) \otimes \omega_{\GG(k,N)} \otimes \bigoplus_{\lambda \in P(N-k,k)} (\s_{\lambda}\V)^{\vee} [k(N-k)]  ) \\
	& \cong \bigoplus_{\lambda \in P(N-k,k)}  \pi_{2*}(\s_{\mu}(\CC^N/\V)\otimes \omega_{\GG(k,N)} \otimes (\s_{\lambda}\V)^{\vee} [k(N-k)-|\mu|+d'] ).
	\end{split}
\end{align} where $\pi_{1}, \pi_{2}$ are the natural projection from $\GG(k,N) \times \GG(0,N)$ to $\GG(k,N)$, $\GG(0,N)$ respectively.

Let $\mu=(\mu_{1},...,\mu_{N-k})$ and $\lambda=(\lambda_{1},..,\lambda_{k})$ with $0 \leq \mu_{N-k} \leq ... \leq \mu_{1} \leq k$ and $ 0 \leq \lambda_{k} \leq ... \leq \lambda_{1} \leq N-k$. Since $\omega_{\GG(k,N)}=\tdet (\V)^{-N+k} \otimes \tdet (\CC^N/\V)^{k}$, we have 
\begin{align*}
&\s_{\mu}(\CC^N/\V) \otimes \omega_{\GG(k,N)} \otimes (\s_{\lambda}\V)^{\vee} 
\cong \s_{\mu}(\CC^N/\V) \otimes \tdet (\V)^{N-k} \otimes \tdet (\CC^N/\V)^{-k}  \otimes (\s_{\lambda}\V)^{\vee} \\
& \cong \s_{(\mu_{1},...,\mu_{N-k})}(\CC^N/\V) \otimes \tdet (\V)^{N-k} \otimes \tdet (\CC^N/\V)^{-k}  \otimes \s_{(-\lambda_{k},..,-\lambda_{1})}\V \\
& \cong \s_{(\mu_{1}-k,...,\mu_{N-k}-k)}(\CC^N/\V) \otimes   \s_{(-\lambda_{k}+N-k,..,-\lambda_{1}+N-k)}\V \\
& \cong \s_{(-\mu_{N-k}+k,...,-\mu_{1}+k)}(\CC^N/\V)^{\vee} \otimes   \s_{(-\lambda_{k}+N-k,..,-\lambda_{1}+N-k)}\V.
\end{align*}

Since $\pi_{2}$ is projection from $\GG(k,N)$ to a point, $ \pi_{2*}(\s_{\mu}(\CC^N/\V) \otimes \omega_{\GG(k,N)} \otimes (\s_{\lambda}\V)^{\vee} [k(N-k)-|\mu|+d'] )$ is calculating the sheaf cohomology. More precisely, 
\begin{align*}
&\pi_{2*}(\s_{\mu}(\CC^N/\V) \otimes \omega_{\GG(k,N)} \otimes (\s_{\lambda}\V)^{\vee} [k(N-k)-|\mu|+d'] ) \\
& \cong H^{*}(\GG(k,N), \s_{\mu}(\CC^N/\V) \otimes \omega_{\GG(k,N)} \otimes (\s_{\lambda}\V)^{\vee} [k(N-k)-|\mu|+d'] ) \\
&\cong H^{*}(\GG(k,N), \s_{(-\mu_{N-k}+k,...,-\mu_{1}+k)}(\CC^N/\V)^{\vee} \otimes   \s_{(-\lambda_{k}+N-k,..,-\lambda_{1}+N-k)}\V [k(N-k)-|\mu|+d'] )
\end{align*}

Note that we have $0 \leq -\mu_{1}+k \leq ... \leq -\mu_{N-k}+k \leq k$ and $ 0 \leq -\lambda_{1}+N-k \leq ... \leq -\lambda_{k}+N-k \leq N-k$, so by Lemma \ref{lemma 3}, we conclude that 
\begin{align*}
&H^{i}(\GG(k,N), \s_{(-\mu_{N-k}+k,...,-\mu_{1}+k)}(\CC^N/\V)^{\vee} \otimes   \s_{(-\lambda_{k}+N-k,..,-\lambda_{1}+N-k)}\V [k(N-k)-|\mu|+d'] ) \\
&= \begin{cases}
	0 & \text{for all} \ i, \ \text{if} \ (-\lambda_{k}+N-k,..,-\lambda_{1}+N-k) \neq (-\mu_{N-k}+k,...,-\mu_{1}+k)^{*} \\
	0 &  \text{for all} \ i \neq -d',  \text{if} \ (-\lambda_{k}+N-k,..,-\lambda_{1}+N-k)= (-\mu_{N-k}+k,...,-\mu_{1}+k)^{*} \\
	\CC &  \text{for all} \ i=-d',  \text{if} \ (-\lambda_{k}+N-k,..,-\lambda_{1}+N-k)= (-\mu_{N-k}+k,...,-\mu_{1}+k)^{*}. \\
\end{cases}
\end{align*}

This implies that (\ref{eq 39}) is isomorphic to $\bigoplus_{\lambda} \CC[d'] \in \Dd^{\leq 0}(0,N)$, where $d' \geq 0$ and the sum is over all $\lambda=(\lambda_{1},..,\lambda_{k})$ such that $(-\lambda_{k}+N-k,..,-\lambda_{1}+N-k)= (-\mu_{N-k}+k,...,-\mu_{1}+k)^{*}$. Thus $(\varphi'_{k})^{R}(\s_{\mu}\CC^N/\V[-|\mu|+d']) \in  \Dd^{\geq 0}(0,N)$.

Next, we show that $\s_{\lambda}\V[d] \in \Dd_{act}^{\geq 0}(k,N-k)$ for all $d \leq 0$. Similarly, we have to show that $(\varphi'_{k})^{R}(\s_{\lambda}\V[d]) \in  \Dd^{\geq 0}(0,N)$. Again, by definition of the FM transform and using the same argument to simplify, we have 
\begin{align} \label{eq 40}
	\begin{split}
		&(\varphi'_{k})^{R}(\s_{\lambda}\V[d])=\pi_{2*}(\pi^{*}_{1}(\s_{\lambda}\V[d] \otimes \omega_{\GG(k,N)} \otimes \bigoplus_{\lambda' \in P(N-k,k)} (\s_{\lambda'}\V)^{\vee} [k(N-k)]  ) \\
		& \cong \bigoplus_{\lambda' \in P(N-k,k)}  \pi_{2*}(\s_{\lambda}\V \otimes \omega_{\GG(k,N)} \otimes (\s_{\lambda'}\V)^{\vee} [k(N-k)+d] ) \\
		& \cong \bigoplus_{\lambda' \in P(N-k,k)}  H^{*}(\GG(k,N),\s_{(\lambda_{1},...,\lambda_{k})}\V \otimes \s_{(N-k-\lambda'_{k},...,N-k-\lambda'_{1})}\V \otimes \s_{(k,...,k)}(\CC^N/\V)^{\vee} [k(N-k)+d] ).
	\end{split}
\end{align} where $0 \leq \lambda_{k} \leq ... \leq \lambda_{1} \leq N-k$ and $0 \leq \lambda'_{k} \leq ... \leq \lambda'_{1} \leq N-k$, thus $0 \leq N-k-\lambda'_{1} \leq ... \leq N-k-\lambda'_{k} \leq N-k$.

If we denote $(N-k-\lambda'_{k},...,N-k-\lambda'_{1})$ by $\lambda''$, then we have to calculate $\s_{\lambda}\V \otimes \s_{\lambda''}\V$, which will be given by the Littlewood-Richardson rule. More precisely, we can write 
\[
\s_{\lambda}\V \otimes \s_{\lambda''}\V \cong \bigoplus_{\nu} (\s_{\nu}\V)^{\oplus c^{\lambda,\lambda''}_{\nu}}
\] where $\nu$ are certain Young diagram that obtained by adding cells of $\lambda''$ to $\lambda$, and $c^{\lambda,\lambda''}_{\nu}$ is the multiplicity called the Littlewood-Richardson coefficients.

So we have to calculate $H^{*}(\GG(k,N),\s_{\nu}\V \otimes \s_{(k,...,k)}(\CC^N/\V)^{\vee} [k(N-k)+d] )$. Note that since $0 \leq \lambda_{k} \leq ... \leq \lambda_{1} \leq N-k$ and $0 \leq N-k-\lambda'_{1} \leq ... \leq N-k-\lambda'_{k} \leq N-k$,  the possible range for $\nu=(\nu_{1},...,\nu_{k})$ are 
$0 \leq \nu_{k} \leq .... \leq \nu_{1} \leq 2(N-k)$.
 
By Lemma \ref{lemma 3}, we conclude that $H^{i}(\GG(k,N),\s_{\nu}\V \otimes \s_{(k,...,k)}(\CC^N/\V)^{\vee} [k(N-k)+d] )$ only nonzero when $i=-d$. Thus (\ref{eq 40}) becomes $\bigoplus_{i \in I} \CC^{n_{i}}[d]$ for some finite index set $I$ and positive integers $n_{i}$, which is in $\Dd^{\geq 0}(0,N)$. Hence we show that $(\varphi'_{k})^{R}(\s_{\lambda}\V[d]) \in  \Dd^{\geq 0}(0,N)$.
\end{proof}

\section{Passage to the Grothendieck groups} \label{Section 6}

One of the property for full exceptional collections in a triangulated category $\Dd$ is that, when pass to the K-theory (Grothendieck group), by definition, they give semiorthogonal basis with respect to the bilinear Euler form $\sum_{i \in \ZZ} (-1)^i \dim_{\CC} \Hom_{\Dd} ( - , - [i] )$.

From Corollary \ref{Corollary 2}, we have a full exceptional collection $R_{(k,N-k)}$ for $\Dd^b(\GG(k,N)))$. So when pass to the K-theory, we obtain a basis for $K(\GG(k,N))$. On the other hand, from Theorem \ref{Theorem 3}, we also have a categorical action of $\dot{\Uu}_{0,N}(L\SL_{2})$ on $\bigoplus_{k} \Dd^b(\GG(k,N))$. Thus, it induces an action of $\dot{\Uu}_{0,N}(L\SL_{2})$ on $\bigoplus_{k} K(\GG(k,N))$.

In this section, as an application, we  calculate the matrix 
coefficients for the generators $e_{r}1_{(k,N-k)}$, $f_{s}1_{(k,N-k)}$ of $\dot{\Uu}_{0,N}(L\SL_{2})$. In fact, the result we obtain holds for all $r, s \in \ZZ$ since the generators $e_{r}1_{(k,N-k)}$, $f_{s}1_{(k,N-k)}$ can be defined for all  $r, s \in \ZZ$.

This would not be done by direct calculation using the definition of the generators $e_{r}1_{(k,N-k)}, \ f_{s}1_{(k,N-k)}$, i.e. K-theoretical Fourier-Mukai transformation. Note that $e_{r}1_{(k,N-k)}, \ f_{s}1_{(k,N-k)}$ can be defined as 
$e1_{(k,N-k)}$, $f1_{(k,N-k)}$ up to some conjugation of determinant line bundles (see \ref{subsection 6.2} for details). So we need to calculate $e1_{(k,N-k)}$, $f1_{(k,N-k)}$ first. This is done by calculated at the categorical level via using Theorem \ref{Theorem 5} and convolution of kernels, then pass to the K-theory to obtain the result.

\subsection{The case $r=s=0$} \label{subsection 6.1}



In this subsection, we calculate the action of $e1_{(k,N-k)}, \ f1_{(k,N-k)}$ on the basis elements from the exceptional collection $\R_{(k,N-k)}$.

For each $k$, recall that the objects in $\R_{(k,N-k)}$ is of the form $\s_{\lambda}\V$ with $\lambda=(\lambda_{1},...,\lambda_{k})$ and $0 \leq \lambda_{k} \leq ... \leq \lambda_{1} \leq N-k$. By Corollary \ref{Corollary 6}, we have $\s_{\lambda}\V = \Ff_{\lambda_{1}} \ast ... \ast \Ff_{\lambda_{k}} \bo_{(0,N)}$. We have to understand how the functors ${\E}\bo_{(k,N-k)}$, ${\F}\bo_{(k,N-k)}$  acting on the objects of $\R_{(k,N-k)}$ for all $k$.



From Theorem \ref{Theorem 3}, we know that the action of $\dot{\Uu}_{0,N}(L\SL_{2})$ on $\bigoplus_{k}\Dd^b(\GG(k,N))$ is by using FM transformations, i.e.  ${\E}\bo_{(k,N-k)}$, ${\F}\bo_{(k,N-k)}$ are FM transformations with FM kernels  ${\Ee}\bo_{(k,N-k)}$, ${\Ff}\bo_{(k,N-k)}$. Thus to study the action of ${\E}\bo_{(k,N-k)}$, ${\F}\bo_{(k,N-k)}$ on those objects is to study the convolution of kernels and use Definition \ref{definition 2} to deduce the results.

\begin{proposition} \label{Proposition 3} 
Let $\lambda=(\lambda_{1},...,\lambda_{k})$ with $0 \leq \lambda_{k} \leq ... \leq \lambda_{1} \leq N-k$. We have the following isomorphism.
\begin{align}
\begin{split}
&{\E}\bo_{(k,N-k)}(\s_{\lambda}\V) \cong \Ee \ast  \Ff_{\lambda_{1}} \ast ... \ast \Ff_{\lambda_{k}} \bo_{(0,N)}  \\
& \cong \begin{cases}
0 & \ \text{if} \  1 \leq \lambda_{k} \leq N-k  \\ 
\Ff_{\lambda_{1}} \ast ... \ast \Ff_{\lambda_{k-1}} \bo_{(0,N)}  & \ \text{if} \   \lambda_{k}=0 
\end{cases}  \label{eq 18}  
\end{split}
\\
\begin{split}
&{\F}\bo_{(k,N-k)}(\s_{\lambda}\V) \cong \Ff \ast  \Ff_{\lambda_{1}} \ast ... \ast \Ff_{\lambda_{k}} \bo_{(0,N)} \\
& \cong \begin{cases}
	0 & \ \text{if} \  \lambda_{i}=i \  \text{for some} \ i  \\
	\Ff_{\lambda_{1}-1} \ast ... \ast  \Ff_{\lambda_{t}-1} \ast \Ff_{t} \ast \Ff_{\lambda_{t+1}}  ... \ast \Ff_{\lambda_{k}}[-t] \bo_{(0,N)}  & \ \parbox[t]{.3\textwidth}{$\exists$ ! $0 \leq t \leq k$ s.t.
		$\lambda_{i} \geq i+1$ for $1 \leq i \leq t$ and $\lambda_{t+1} \leq t$}
\end{cases}     \label{eq 19}
\end{split}
\end{align}
\end{proposition}

\begin{proof}
We use Definition \ref{definition 2} of the categorical action to help us prove this proposition. 

For (\ref{eq 18}), we have to calculate $ \Ee \ast  \Ff_{\lambda_{1}} \ast ... \ast \Ff_{\lambda_{k}} \bo_{(0,N)} $. Note that since $0 \leq \lambda_{k} \leq ... \leq \lambda_{1} \leq N-k$, we have  $\Ee \ast  \Ff_{\lambda_{1}} \bo_{(k-1, N-k+1)} \cong \Ff_{\lambda_{1}} \ast \Ee \bo_{(k-1, N-k+1)}$, which is by relation (16) (e) in Definition \ref{definition 2}. Similarly, we have $\Ee \ast  \Ff_{\lambda_{2}} \bo_{(k-2, N-k+2)} \cong \Ff_{\lambda_{1}} \ast \Ee \bo_{(k-1, N-k+2)}$ and so on. 

Eventually, we would end up with 
\begin{equation} \label{eq 20}
 \Ee \ast  \Ff_{\lambda_{1}} \ast ... \ast \Ff_{\lambda_{k}} \bo_{(0,N)}  \cong ... \cong   \Ff_{\lambda_{1}} \ast ... \ast \Ee \ast  \Ff_{\lambda_{k}} \bo_{(0,N)} .
\end{equation}

If $1 \leq \lambda_{k} \leq N-k$, then $\Ee \ast  \Ff_{\lambda_{k}} \bo_{(0,N)} \cong  \Ff_{\lambda_{k}} \ast \Ee \bo_{(0,N)}$ and so $(\ref{eq 20}) \cong 0$. Otherwise, we have $\lambda_{k}=0$ and so $\Ee \ast  \Ff_{\lambda_{k}} \bo_{(0,N)} \cong \Psi^{-}[-1]\bo_{(0,N)}$ from relation (16) (c) in Definition \ref{definition 2}. Since $\Psi^{-}$ is given by a FM kernel from Definition \ref{definition 3}, we have $\Psi^{-}[-1]\bo_{(0,N)} \cong \bo_{(0,N)}$ and thus  $\Ff_{\lambda_{1}} \ast ... \ast \Ee \ast  \Ff_{\lambda_{k}} \bo_{(0,N)}  \cong \Ff_{\lambda_{1}} \ast ... \ast \Ff_{\lambda_{k-1}} \bo_{(0,N)} $.

For (\ref{eq 19}), we use relation (10)(a) in Definition \ref{definition 2}.

If $\lambda_{i}=i$ for some $1 \leq i \leq k$, then we have $i=\lambda_{i} \leq \lambda_{i-1} \leq ... \leq \lambda_{1}$. This implies that 
\begin{align*}
\Ff \ast  \Ff_{\lambda_{1}} \ast ... \ast \Ff_{\lambda_{k}} \bo_{(0,N)} & \cong \Ff_{\lambda_{1}-1} \ast \Ff_{1} \ast ...  \ast \Ff_{\lambda_{k}} \bo_{(0,N)}[-1] \\
&\cong ...  \\
&\cong  \Ff_{\lambda_{1}-1} \ast ... \Ff_{\lambda_{i-1}-1}  \ast \Ff_{i-1} \ast \Ff_{\lambda_{i}} ...\ast \Ff_{\lambda_{k}} \bo_{(0,N)}[-i+1] \\
& \cong 0
\end{align*}

Otherwise, there exist an unique $0 \leq t \leq k$ such that 
$\lambda_{i} \geq i+1$ for  $1 \leq i \leq t$  and $\lambda_{t+1} \leq t$.  More precisely, we have the following cases. 

When $t=0$, this means that $\lambda_{1} \leq 0$ and thus $\lambda_{i}=0$ for all $i$. So $\Ff \ast  \Ff_{\lambda_{1}} \ast ... \ast \Ff_{\lambda_{k}} \bo_{(0,N)} \cong  \Ff \ast  \Ff\ast ... \ast \Ff \bo_{(0,N)}$. 

When $1 \leq t \leq k-1$, we have $\lambda_{1} \geq 2, \ \lambda_{2} \geq 3, \ ..., \ \lambda_{t} \geq t+1$ and $\lambda_{t+1} \leq t$. This means that $\lambda_{k} \leq .. \leq \lambda_{t+1} \leq t$. So by relation (10)(a) in Definition \ref{definition 2}, we obtain 
\begin{equation*}
\Ff \ast  \Ff_{\lambda_{1}} \ast ... \ast \Ff_{\lambda_{k}} \bo_{(0,N)} \cong \Ff_{\lambda_{1}-1} \ast ... \ast  \Ff_{\lambda_{t}-1} \ast \Ff_{t} \ast \Ff_{\lambda_{t+1}}  ... \ast \Ff_{\lambda_{k}}[-t] \bo_{(0,N)} 
\end{equation*} where $(\lambda_{1}-1,...,\lambda_{t}-1,t,\lambda_{t+1},...,\lambda_{k}) \in P(N-k-1,k+1)$.

Finally, when $t=k$, we have $\lambda_{i} \geq i+1$ for all $1 \leq i \leq k$, so  $\Ff \ast  \Ff_{\lambda_{1}} \ast ... \ast \Ff_{\lambda_{k}} \bo_{(0,N)} \cong \Ff_{\lambda_{1}-1} \ast  ... \ast \Ff_{\lambda_{k}-1} \ast \Ff_{k}[-k] \bo_{(0,N)}$ and $(\lambda_{1}-1,...,\lambda_{k}-1,k) \in P(N-k-1,k+1)$.

\end{proof}

As a corollary, we obtain how the generators $e1_{(k,N-k)}$, $f1_{(k,N-k)}$ act on the basis 	$[R_{(k,N-k)}]=\{[\s_{\lambda}\V] \ | \  \lambda \in P(N-k,k) \}$.

\begin{corollary} \label{Corollary 7}
For $\lambda= (\lambda_{1},...,\lambda_{k}) \in P(N-k,k)$ and $\V, \ \V', \ \V''$ be the tautological bundle of rank $k,k-1,k+1$ on $\GG(k,N)$, $\GG(k-1,N)$, $\GG(k+1,N)$, respectively. We have 
\begin{align*}
e1_{(k,N-k)}([\s_{\lambda}(\V)]) = \begin{cases}
0 & \text{if} \ 1\leq  \lambda_{k} \leq N-k \\ 
[\s_{(\lambda_{1},...,\lambda_{k-1})}(\V') ] & \text{if} \ \lambda_{k}=0 \\ 
\end{cases}
\end{align*}
\begin{align*}
f1_{(k,N-k)}([\s_{\lambda}(\V)]) = \begin{cases}
0 & \ \text{if} \  \lambda_{i}=i \  \text{for some} \ i  \\
(-1)^{r}[\s_{(\lambda_{1}-1,...,\lambda_{t}-1,t,\lambda_{t+1},...\lambda_{k})}(\V'') ] & \ \parbox[t]{.3\textwidth}{$\exists$ ! $0 \leq t \leq k$ s.t.
	$\lambda_{i} \geq i+1$ for $1 \leq i \leq t$ and $\lambda_{t+1} \leq t$}
\end{cases}
\end{align*}
\end{corollary}

\subsection{For general $r,s \neq 0$}  \label{subsection 6.2}

Let $\V, \ \V', \ \V''$ be the tautological bundle of rank $k,k-1,k+1$ on $\GG(k,N)$, $\GG(k-1,N)$, $\GG(k+1,N)$, respectively. Then we have the basis
\begin{align*}
&[R_{(k,N-k)}]=\{[\s_{\lambda}\V] \ | \  \lambda \in P(N-k,k) \}, \\
&[R_{(k-1,N-k+1)}]=\{[\s_{\lambda'}\V'] \ | \  \lambda' \in P(N-k+1,k-1) \}, \\
&[R_{(k+1,N-k-1)}]=\{[\s_{\lambda''}\V''] \ | \  \lambda'' \in P(N-k-1,k+1) \}
\end{align*}
for their Grothendieck groups $K(\GG(k,N))$, $K(\GG(k-1,N))$, $K(\GG(k+1,N))$, respectively.

Let us fix some notations, we  denote $[\lambda]$ to be the basis element $[\s_{\lambda}\V]$ for $\lambda \in P(N-k,k)$. Similarly for $[\lambda']$ and $[\lambda'']$. Also, for $\lambda=(\lambda_{1},...,\lambda_{k}) \in P(N-k,k)$, we will denote $\lambda+n$ for $(\lambda_{1}+n,...,\lambda_{k}+n)$ for all $n \in \ZZ$ and $\lambda^*=(\lambda_{1}^{*},...,\lambda_{N-k}^{*}) \in P(k,N-k)$ for its conjugate partition. Finally, we use $\gtdim$ for the notation of graded dimension, i.e., 
\begin{equation*}
	\gtdim R\Hom(\Xx, \Yy):=\sum_{i \in \ZZ} (-1)^{i} \tdim \Ext^{i}(\Xx,\Yy)
\end{equation*} for all $\Xx, \Yy \in \Dd^b(\GG(k,N))$.

The goal of this subsection is to calculate the matrix coefficients for $e_{r}1_{(k,N-k)}, \ f_{s}1_{(k,N-k)}$; i.e.,  $[e_{r}1_{(k,N-k)}]_{[\lambda],[\lambda']}$ and $[f_{s}1_{(k,N-k)}]_{[\lambda],[\lambda'']}$ for $r,s \neq 0$.

Consider $e_{r}1_{(k,N-k)}$ first, let us recall that its action at categorical level is given by the following correspondence 
\begin{equation*}
\xymatrix{ 
	&&Fl(k-1,k)=\{0 \overset{k-1}{\subset} V' \overset{1}{\subset} V \overset{N-k}{\subset} \CC^N \} 
	\ar[ld]_{p_1} \ar[rd]^{p_2}   \\
	& \GG(k,N)  && \GG(k-1,N)
}
\end{equation*}  and ${\E}_{r}1_{(k,N-k)}:=p_{2*}(p^{*}_{1}\otimes (\V/\V')^{r})$. Let us define the functor $\G:\Dd^b(\GG(k,N)) \rightarrow \Dd^b(\GG(k,N))$ to be given by tensoring by the line bundle $\det(\V)$ and the corresponding map $g:K(\GG(k,N)) \rightarrow K(\GG(k,N))$ after passing to the K-theory. Then we have ${\E}_{r}1_{(k,N-k)}={\G}^{-r}{\E}{\G}^{r}1_{(k,N-k)}$ and $e_{r}1_{(k,N-k)}=g^{-r}eg^{r}1_{(k,N-k)}$ for all $r \in \ZZ$. Similarly we have $f_{s}1_{(k,N-k)}=g^{s}eg^{-s}1_{(k,N-k)}$ for all $s \in \ZZ$.

Before we give the result for the matrix coefficient of $[e_{r}1_{(k,N-k)}]_{[\lambda],[\lambda']}$ , we have to introduce more notations. 

\begin{definition} \label{definition 4} 
For fixed $r \neq 0$, $\lambda \in P(N-k,k)$, $\lambda' \in P(N-k+1,k-1)$, we define
\begin{align*}
&x_{\mu}:=(\lambda_{k}+r,...,\lambda_{1}+r,-\mu^{*}_{1},...,-\mu^{*}_{N-k}) \\ &y_{\mu}:=(\mu_{k-1}-r,...,\mu_{1}-r,-\lambda_{1}^{'*},...,-\lambda^{'*}_{N-k+1})
\end{align*} 
for all $\mu \in P(N-k,k)$ with $\mu_{k}=0$.
\end{definition}

\begin{definition}  \label{definition 5} 
For fixed $r \neq 0$, $\lambda \in P(N-k,k)$, $\lambda' \in P(N-k+1,k-1)$, we define the following three subsets of $P(N-k,k)$.
\begin{align*}
	P_{1}(N-k,k)=\bigg\{ \mu \in P(N-k,k)\ : \ \begin{split}
		&\mu_{k}=0,\  \mu^{*}_{1} \leq -\lambda_{1}-r, \ \mu_{1} > r-\lambda_{1}^{'*} \\ 
		& \exists ! \ w \in W \ \text{s.t.} \ w \star (-y_{\mu}) \ \text{is dominant} 
	\end{split}
	\bigg\}
\end{align*}
\begin{align*}
	P_{2}(N-k,k)=\bigg\{ \mu \in P(N-k,k)\ : \ \begin{split}
		&\mu_{k}=0,\ \mu^{*}_{1} > -\lambda_{1}-r, \ \mu_{1} \leq r-\lambda_{1}^{'*} \\ & \exists ! \ w' \in W \ \text{s.t.} \ w' \star (-x_{\mu}) \ \text{is dominant} 
	\end{split}
	\bigg\}
\end{align*}
\begin{align*}
	P_{3}(N-k,k)=\bigg\{ \mu \in P(N-k,k)\ : \ \begin{split}
		&\mu_{k}=0,\ \mu^{*}_{1} > -\lambda_{1}-r, \ \mu_{1} > r-\lambda_{1}^{'*} \\ & \exists ! \ w_{1}, \ w_{2} \in W \ \text{s.t.} \ w_{1} \star (-x_{\mu}), w_{2} \star (-y_{\mu}) \ \text{are dominant} 
	\end{split}
	\bigg\}
\end{align*}
\end{definition}

First, the matrix coefficients $[e_{r}1_{(k,N-k)}]_{[\lambda],[\lambda']}$ is given in the following Theorem.

\begin{theorem} \label{Theorem 10} 
For $r \neq 0$, we have 
\begin{align*}
\begin{split}
[e_{r}1_{(k,N-k)}]_{[\lambda],[\lambda']}=&\sum_{\mu \in P_{1}(N-k,k)}  (-1)^{|\mu|+|\lambda'|+l(w)} \dim (\s_{-x_{\mu}}\CC^N)^{\vee} \dim  (\s_{w \star (-y_{\mu})}\CC^N )^{\vee} \\
&+\sum_{ \mu \in P_{2}(N-k,k)}  (-1)^{|\mu|+|\lambda'|+l(w')} \dim (\s_{w' \star (-x_{\mu})}\CC^N)^{\vee} \dim  (\s_{-y_{\mu}}\CC^N )^{\vee} \\
&+\sum_{ \mu \in P_{3}(N-k,k)}  (-1)^{|\mu|+|\lambda'|+l(w_{1})+l(w_{2})} \dim (\s_{w_{1} \star (-x_{\mu})}\CC^N)^{\vee} \dim  (\s_{w_{2} \star (-y_{\mu})}\CC^N )^{\vee} 
\end{split}
\end{align*}
where $x_{\mu}, y_{\mu}$ and $P_{1}(k,N-k), P_{2}(N-k,k), P_{3}(N-k,k)$ are defined in Definition \ref{definition 4}, \ref{definition 5}.
\end{theorem}

\begin{proof}
To calculate $[e_{r}1_{(k,N-k)}]_{[\lambda],[\lambda']}$, we have to calculate $e_{r}1_{(k,N-k)}([\s_{\lambda}\V])$. By above argument, we get 
\begin{equation} \label{eq 45}
\begin{split}
e_{r}1_{(k,N-k)}([\s_{\lambda}\V])&=g^{-r}eg^{r}1_{(k,N-k)}([\s_{\lambda}\V]) \\
&=g^{-r}e1_{(k,N-k)}([\s_{\lambda}\V\otimes \det(\V)^{r}])=g^{-r}e1_{(k,N-k)}([\s_{\lambda+r}\V]).
\end{split} 
\end{equation}

From (\ref{eq 44}), we know that the dual exceptional collection is given by \begin{equation*}
\R'_{(k,N-k)} = \{\s_{\mu^{*}}(\CC^N/\V)[-|\mu|] \ | \ \mu \in P(N-k,k) \}
\end{equation*} which implies that $\{ (-1)^{|\mu|} [\s_{\mu^{*}}(\CC^N/\V)] \ | \ \mu \in P(N-k,k) \}\}$ is the dual basis for $\{[\s_{\lambda}\V] \ | \  \lambda \in P(N-k,k) \}$ under the bilinear Euler form $\sum_{i \in \ZZ} (-1)^i \dim_{\CC} \Hom_{\Dd} ( - , - [i] )$.

So $[\s_{\lambda+r}(\V)] \in K(\GG(k,N))$ can be written as 
\begin{equation*}
[\s_{\lambda+r}(\V)]=\sum_{\mu \in P(N-k,k) } (-1)^{|\mu|} \gtdim R\Hom(\s_{\mu^{*}}(\CC^N/\V), \s_{\lambda+r}(\V)) [\s_{\mu}\V].
\end{equation*}

So (\ref{eq 45}) becomes 
\begin{align}
\begin{split}
& g^{-r}e1_{(k,N-k)} (\sum_{\mu \in P(N-k,k) } (-1)^{|\mu|} \gtdim R\Hom(\s_{\mu^{*}}(\CC^N/\V), \s_{\lambda+r}\V) [\s_{\mu}\V]) \\
&= g^{-r}(\sum_{\mu \in P(N-k,k) } (-1)^{|\mu|} \gtdim R\Hom(\s_{\mu^{*}}(\CC^N/\V), \s_{\lambda+r}\V)e1_{(k,N-k)}([\s_{\mu}\V])) \\
&=g^{-r}(\sum_{ \substack{\mu \in P(N-k,k) \\ with \ \mu_{k}=0}}  (-1)^{|\mu|} \gtdim R\Hom(\s_{\mu^{*}}(\CC^N/\V), \s_{\lambda+r}\V)[\s_{\mu}\V']) \ \text{by Corollary \ref{Corollary 7}} \\
&=\sum_{ \substack{\mu \in P(N-k,k) \\ with \ \mu_{k}=0}} (-1)^{|\mu|} \gtdim R\Hom(\s_{\mu^{*}}(\CC^N/\V), \s_{\lambda+r}\V)[\s_{\mu-r}\V']. \label{eq 46}
\end{split}
\end{align}
 
Now for  $[\s_{\mu-r}\V'] \in K(\GG(k-1,N))$, we also have the dual basis $\R'_{(k-1,N-k+1)}$ similar to (\ref{eq 44}). So we obtain 
 \begin{equation*}
 [\s_{\mu-r}\V']=\sum_{\lambda' \in P(N-k+1,k-1) } (-1)^{|\lambda'|} \gtdim R\Hom(\s_{\lambda'^{*}}(\CC^N/\V'), \s_{\mu-r}\V') [\s_{\lambda'}\V'].
 \end{equation*}

Combining them together, we have (\ref{eq 46}) becomes 
\begin{equation*}
\sum_{ \substack{\mu \in P(N-k,k) \\ with \ \mu_{k}=0}}  \sum_{\lambda' \in P(N-k+1,k-1) }(-1)^{|\mu|+|\lambda'|} \gtdim R\Hom(\s_{\mu^{*}}(\CC^N/\V), \s_{\lambda+r}\V) \gtdim R\Hom(\s_{\lambda'^{*}}(\CC^N/\V'), \s_{\mu-r}\V') [\s_{\lambda'}\V'].
\end{equation*}

Thus the matrix coefficient $[e_{r}1_{(k,N-k)}]_{[\lambda],[\lambda']}$ is given by 
\begin{equation} \label{eq 47}
	\sum_{ \substack{\mu \in P(N-k,k) \\ with \ \mu_{k}=0}}  (-1)^{|\mu|+|\lambda'|} \gtdim R\Hom(\s_{\mu^{*}}(\CC^N/\V), \s_{\lambda+r}\V) \gtdim R\Hom(\s_{\lambda'^{*}}(\CC^N/\V'), \s_{\mu-r}\V') 
\end{equation}

To know (\ref{eq 47}), we have to calculate $R\Hom(\s_{\mu^{*}}(\CC^N/\V), \s_{\lambda+r}\V)$ and $R\Hom(\s_{\lambda'^{*}}(\CC^N/\V'), \s_{\mu-r}\V')$. By using the cohomological computation from 2.5 of \cite{Kap85}, we have the following 
\begin{align*}
R\Hom(\s_{\mu^{*}}(\CC^N/\V), \s_{\lambda+r}\V) &\cong R\Hom(\Oo_{\GG(k,N)}, \s_{\lambda+r}\V\otimes (\s_{\mu^{*}}(\CC^N/\V))^{\vee} ) \\
&\cong H^*(\GG(k,N), \s_{\lambda+r}\V\otimes (\s_{\mu^{*}}(\CC^N/\V))^{\vee})  \cong H^*(Fl,\LL_{x_{\mu}})  \\
R\Hom(\s_{\lambda'^{*}}(\CC^N/\V'), \s_{\mu-r}\V') &\cong R\Hom(\Oo_{\GG(k-1,N)}, \s_{\mu-r}\V'\otimes (\s_{\lambda'^{*}}(\CC^N/\V'))^{\vee} ) \\
&\cong H^*(\GG(k-1,N), \s_{\mu-r}\V'\otimes (\s_{\lambda'^{*}}(\CC^N/\V'))^{\vee})  \cong H^*(Fl,\LL_{y_{\mu}}) 
\end{align*} where $Fl$ is the full flag variety and $\LL_{x}$, $\LL_{y}$ are the line bundles given by the characters $x_{\mu}=(\lambda_{k}+r,...,\lambda_{1}+r,-\mu^{*}_{1},...,-\mu^{*}_{N-k})$ and $y_{\mu}=(\mu_{k-1}-r,...,\mu_{1}-r,-\lambda_{1}^{'*},...,-\lambda^{'*}_{N-k+1})$ respectively.

The summation term in (\ref{eq 47}) is non-zero only when both $R\Hom(\s_{\mu^{*}}(\CC^N/\V), \s_{\lambda+r}\V)$ and $R\Hom(\s_{\lambda'^{*}}(\CC^N/\V), \s_{\mu-r}\V')$ are non-zero. By the above argument, this is equivalent to $H^*(Fl,\LL_{x_{\mu}}) \neq 0$ and $H^*(Fl,\LL_{y_{\mu}}) \neq 0$. 

By using the Borel-Weil-Bott Theorem \ref{Theorem 4}, there are 4 conditions for $H^*(Fl,\LL_{x_{\mu}}) \neq 0$ and $H^*(Fl,\LL_{y_{\mu}}) \neq 0$. They are 
\begin{enumerate}
	\item Both $-x_{\mu}$ and $-y_{\mu}$ are dominant
	\item $-x_{\mu}$ is dominant, $-y_{\mu}$ is not dominant but $\exists ! \ w \in W$ such that $w \star (-y_{\mu})$ is dominant.
	\item $-y_{\mu}$ is dominant, $-x_{\mu}$ is not dominant but $\exists ! \ w \in W$ such that $w \star (-x_{\mu})$ is dominant.
	\item Both $-x_{\mu}$ and $-y_{\mu}$ are not dominant , but $\exists ! \ w_{1} \in W$ $\exists ! \ w_{2} \in W$ such that $w_{1} \star (-x_{\mu})$ and $w_{2} \star (-y_{\mu})$ are dominant.
\end{enumerate}

Consider the first case (1), we have $-x_{\mu}=(-\lambda_{k}-r,...,-\lambda_{1}-r,\mu^{*}_{1},...,\mu^{*}_{N-k})$ and $-y_{\mu}=(-\mu_{k-1}+r,...,-\mu_{1}+r,\lambda_{1}^{'*},...,\lambda^{'*}_{N-k+1})$.

Note that we have 
\begin{align*}
& \lambda=(\lambda_{1},...,\lambda_{k}) \in P(N-k,k)  \ \text{with} \ 0 \leq \lambda_{k} \leq ... \leq \lambda_{1} \leq N-k \\
& \mu=(\mu_{1},...,\mu_{k-1},0) \in P(N-k,k) \ \text{with} \ 0 \leq \mu_{k-1} \leq ... \leq \mu_{1} \leq N-k \\ 
& \mu^{*}=(\mu^{*}_{1},...,\mu^{*}_{N-k}) \in P(k,N-k) \ \text{with} \  0 \leq \mu^{*}_{N-k} \leq ... \leq \mu^{*}_{1} \leq k \\
& \lambda^{'*}=(\lambda^{'*}_{1},...,\lambda^{'*}_{N-k+1}) \in P(k-1,N-k+1) \ \text{with} \  0 \leq \lambda_{N-k+1}^{'*} \leq ... \leq \lambda_{1}^{'*} \leq k-1
\end{align*}  So $-N+k-r \leq -\lambda_{1}-r \leq ... \leq -\lambda_{k}-r \leq -r$ and $-N+k+r \leq -\mu_{1}+r \leq ... \leq -\mu_{k-1}+r \leq r$. 

Thus, for $-x_{\mu}$ to be dominant, we need $-\lambda_{1}-r \geq \mu^{*}_{1}$. For $-y$ to be dominant, we need $-\mu_{1}+r \geq \lambda_{1}^{'*}$. Since $r \neq 0$, if $r \geq 1$, then $ 0 \leq  \mu^{*}_{1} \leq -\lambda_{1}-r \leq -1$, which is impossible. Similarly, if $r \leq -1$, then $ 0 \leq \lambda_{1}^{'*} \leq -\mu_{1}+r \leq -1$, which is also impossible.  So the $\mu \in P(N-k,k)$ that satisfy condition (1) is empty.

Consider the second case (2), for $-x_{\mu}$ dominant and $-y_{\mu}$ not dominant, we have $-\lambda_{1}-r \geq \mu^{*}_{1}$ and $-\mu_{1}+r < \lambda_{1}^{'*}$ from the analysis in case (1). In this case, if there exist an unique $w \in W$ such that $w \ast (-y_{\mu})$ is dominant, then from Theorem \ref{Theorem 4} we have $H^*(Fl,\LL_{x_{\mu}}) \cong (\s_{-x_{\mu}}\CC^N)^{\vee}$ and $H^*(Fl,\LL_{y_{\mu}}) \cong (\s_{w \star (-y_{\mu})}\CC^N )^{\vee}[-l(w)]$. Those $\mu$ that satisfy this condition are precisely the elements in $P_{1}(N-k,k)$.

Consider the third case (3), for $-x_{\mu}$ not dominant and $-y_{\mu}$ dominant, we have $-\lambda_{1}-r < \mu^{*}_{1}$ and $-\mu_{1}+r \geq \lambda_{1}^{'*}$ from the analysis in case (1). In this case, if there exist an unique $w \in W$ such that $w \ast (-x_{\mu})$ is dominant, then from Theorem \ref{Theorem 4} we have $H^*(Fl,\LL_{x_{\mu}}) \cong (\s_{w \star (-x_{\mu})}\CC^N)^{\vee}[-l(w)]$ and $H^*(Fl,\LL_{y_{\mu}}) \cong (\s_{-y_{\mu}}\CC^N )^{\vee}$. Those $\mu$ that satisfy this condition are precisely the elements in $P_{2}(N-k,k)$.

Finally, consider the fourth case (4), for both $-x_{\mu}$ and $-y_{\mu}$ not dominant, we have $-\lambda_{1}-r < \mu^{*}_{1}$ and $-\mu_{1}+r < \lambda_{1}^{'*}$ from the analysis in case (1). In this case, if there exist unique $w_{1}, \  w_{2} \in W$ such that $w_{1} \star (-x_{\mu})$ and $w_{2} \star (-y_{\mu})$ are dominant, then from Theorem \ref{Theorem 4} we have $H^*(Fl,\LL_{x_{\mu}}) \cong (\s_{w_{1} \star (-x_{\mu})}\CC^N)^{\vee}[-l(w_{1})]$ and $H^*(Fl,\LL_{y_{\mu}}) \cong (\s_{w_{2} \star (-y_{\mu})}\CC^N )^{\vee}[-l(w_{2})]$.  Those $\mu$ that satisfy this condition are precisely the elements in $P_{3}(N-k,k)$.

Hence (\ref{eq 47}) becomes 
\begin{align*}
&\sum_{\mu \in P_{1}(N-k,k)}  (-1)^{|\mu|+|\lambda'|+l(w)} \dim (\s_{-x_{\mu}}\CC^N)^{\vee} \dim  (\s_{w \star (-y_{\mu})}\CC^N )^{\vee} \\
&+\sum_{ \mu \in P_{2}(N-k,k)}  (-1)^{|\mu|+|\lambda'|+l(w')} \dim (\s_{w' \star (-x_{\mu})}\CC^N)^{\vee} \dim  (\s_{-y_{\mu}}\CC^N )^{\vee} \\
&+\sum_{ \mu \in P_{3}(N-k,k)}  (-1)^{|\mu|+|\lambda'|+l(w_{1})+l(w_{2})} \dim (\s_{w_{1} \star (-x_{\mu})}\CC^N)^{\vee} \dim  (\s_{w_{2} \star (-y_{\mu})}\CC^N )^{\vee}.
\end{align*}

\end{proof}

\begin{remark}
Note that the subsets $P_{i}(N-k,k)$ can be empty depends on the value of $r, \lambda_{1}, \lambda^{'*}_{1}$ where $i=1,2,3$, in such case the sum would just be zero.
\end{remark}

\begin{remark}
To know the actual value of the matrix coefficient in Theorem \ref{Theorem 10},  we can use the formula for $\dim \s_{\nu}\CC^N$ of a Young diagram $\nu=(\nu_{1},...,\nu_{N})$ with $0 \leq \nu_{N} \leq ... \leq \nu_{1}$. It is given by 
\begin{equation*}
	\dim \s_{\nu}\CC^N=\prod_{1 \leq i < j \leq N} \frac{\nu_{i}-\nu_{j}+j-i}{j-i}.
\end{equation*}
\end{remark}

Next, we also have to introduce more notations before we give the result for the matrix coefficient $[f_{s}1_{(k,N-k)}]_{[\lambda],[\lambda'']}$.

\begin{definition} \label{definition 6} 
	For fixed $s \neq 0$, $\lambda \in P(N-k,k)$, $\lambda'' \in P(N-k-1,k+1)$, we define
	\begin{align*}
		&x'_{\mu}:=(\lambda_{k}-s,...,\lambda_{1}-s,-\mu^{*}_{1},...,-\mu^{*}_{N-k}) \\ &y'_{\mu}:=(\mu_{k}+s,...,\mu_{t(\mu)+1}+s,t(\mu)+s,\mu_{t(\mu)}-1+s,...,\mu_{1}-1+s,-\lambda_{1}^{''*},...,-\lambda^{''*}_{N-k-1})
	\end{align*} 
for all $\mu \in P(N-k,k)$ such that $\mu_{i} \neq i$ for all $i$, and there exists an unique $0 \leq t(\mu) \leq k$ such that $(\mu_{1}-1,...,\mu_{t(\mu)}-1,t(\mu),\mu_{t(\mu)+1},...,\mu_{k}) \in P(N-k-1,k+1)$.
\end{definition}

\begin{definition}  \label{definition 7} 
	For fixed $s \neq 0$, $\lambda \in P(N-k,k)$, $\lambda'' \in P(N-k-1,k+1)$, we define the following three subsets of $P(N-k,k)$.
	\begin{align*}
		P'_{1}(N-k,k)=\bigg\{ \mu \in P(N-k,k)\ : \ \begin{split}
			&\mu_{i} \neq i \ \forall i, \  \mu^{*}_{1} \leq -\lambda_{1}+s, \ \mu_{1} > 1-s-\lambda_{1}^{''*} \\ 
			& \exists ! \ w \in W \ \text{s.t.} \ w \star (-y'_{\mu}) \ \text{is dominant} 
		\end{split}
		\bigg\}
	\end{align*}
	\begin{align*}
		P'_{2}(N-k,k)=\bigg\{ \mu \in P(N-k,k)\ : \ \begin{split}
			&\mu_{i} \neq i \ \forall i,\ \mu^{*}_{1} > -\lambda_{1}+s, \ \mu_{1} \leq 1-s-\lambda_{1}^{''*} \\ & \exists ! \ w' \in W \ \text{s.t.} \ w' \star (-x'_{\mu}) \ \text{is dominant} 
		\end{split}
		\bigg\}
	\end{align*}
	\begin{align*}
		P'_{3}(N-k,k)=\bigg\{\mu \in P(N-k,k)\ : \ \begin{split}
			&\mu_{i} \neq i \ \forall i,\ \mu^{*}_{1} > -\lambda_{1}+s, \ \mu_{1} > 1-s-\lambda_{1}^{''*} \\ & \exists ! \ w_{1}, w_{2} \in W \ \text{s.t.} \ w_{1} \star (-x'_{\mu}), w_{2} \star (-y'_{\mu}) \ \text{are dominant} 
		\end{split}
		\bigg\}
	\end{align*}
\end{definition}

Then $[f_{s}1_{(k,N-k)}]_{[\lambda],[\lambda'']}$ is given by the following theorem.

\begin{theorem}  \label{Theorem 11}
For $s \neq 0$, we have 
\begin{align*}
\begin{split}
	[f_{s}1_{(k,N-k)}]_{[\lambda],[\lambda'']}&=(-1)^{|\lambda''|}\dim H^*(Fl, \LL_{x'_{0}}) \dim H^*(Fl,\LL_{y'_{0}}) \\
	&+\sum_{\mu \in P'_{1}(N-k,k)}  (-1)^{|\mu|+|\lambda''|+t(\mu) +l(w)} \dim (\s_{-x'_{\mu}}\CC^N)^{\vee} \dim  (\s_{w \star (-y'_{\mu})}\CC^N )^{\vee} \\
	&+\sum_{ \mu \in P'_{2}(N-k,k)}  (-1)^{|\mu|+|\lambda''|+t(\mu) +l(w')} \dim (\s_{w' \star (-x'_{\mu})}\CC^N)^{\vee} \dim  (\s_{-y'_{\mu}}\CC^N )^{\vee} \\
	&+\sum_{ \mu \in P'_{3}(N-k,k)}  (-1)^{|\mu|+|\lambda''|+t(\mu) +l(w_{1})+l(w_{2})} \dim (\s_{w_{1} \star (-x'_{\mu})}\CC^N)^{\vee} \dim  (\s_{w_{2} \star (-y'_{\mu})}\CC^N )^{\vee} 
\end{split}
\end{align*}
where 
\begin{align*}
	&\dim H^*(Fl, \LL_{x'_{0}}) \dim H^*(Fl,\LL_{y'_{0}})\\
	\begin{split}
	& =\begin{cases}
			\dim(\s_{-x'_{0}}\CC^N)^{\vee} \dim(\s_{-y'_{0}}\CC^N)^{\vee} & \text{if} \ s \geq \lambda_{1} \ \text{and} \ s \leq \lambda^{''*}_{1}  \\
			(-1)^{l(w)}\dim(\s_{-x'_{0}}\CC^N)^{\vee}\dim(\s_{w \star (-y'_{0})}\CC^N)^{\vee} &  \parbox[t]{.3\textwidth}{ if $s \geq \lambda_{1}$  and  $s > \lambda^{''*}_{1}$ with $\exists !  \ w \in W$ s.t. $w \star (-y'_{0})$ \ dominant} \\
			(-1)^{l(w')}\dim(\s_{w' \star (-x'_{0})}\CC^N)^{\vee} \dim(\s_{-y'_{0}}\CC^N)^{\vee} &  \parbox[t]{.3\textwidth}{ if $s < \lambda_{1}$  and  $s \leq \lambda^{''*}_{1}$ with $\exists !  \ w' \in W$ s.t. $w' \star (-x'_{0})$ \ dominant} \\
			(-1)^{l(w_{1})+l(w_{2})}\dim(\s_{w_{1} \star (-x'_{0})}\CC^N)^{\vee}\dim(\s_{w_{2} \star (-y'_{0})}\CC^N)^{\vee}&  \parbox[t]{.3\textwidth}{ if $s<\lambda_{1}$  and  $s>\lambda^{''*}_{1}$ with $\exists ! \ w_{1}, w_{2} \in W$ s.t. $w_{1} \star (-x'_{0})$, $w_{2} \star (-y'_{0})$ \ dominant} \\
			0  & \ \text{otherwise}
		\end{cases} 
	\end{split} 
\end{align*} and where $x'_{\mu}, y'_{\mu}$ and $P'_{1}(k,N-k), P'_{2}(N-k,k), P'_{3}(N-k,k)$ are defined in Definition \ref{definition 6}, \ref{definition 7}.
\end{theorem}

\begin{proof}
Similarly to Theorem \ref{Theorem 10}, to calculate $[f_{s}1_{(k,N-k)}]_{[\lambda],[\lambda'']}$, we have to calculate $f_{s}1_{(k,N-k)}([\s_{\lambda}\V])$, which is the following
\begin{align}
	\begin{split}
		&f_{s}1_{(k,N-k)}([\s_{\lambda}\V])=g^{s}fg^{-s}1_{(k,N-k)}([\s_{\lambda}\V]) \\
		&=g^{s}f1_{(k,N-k)}([\s_{\lambda}\V\otimes \det(\V)^{-s}])=g^{s}f1_{(k,N-k)}([\s_{\lambda-s}\V]). \\
		&= g^{s}f1_{(k,N-k)} (\sum_{\mu \in P(N-k,k) } (-1)^{|\mu|} \gtdim R\Hom(\s_{\mu^{*}}(\CC^N/\V), \s_{\lambda-s}(\V)) [\s_{\mu}\V]) \\
		&= g^{s}(\sum_{\mu \in P(N-k,k) } (-1)^{|\mu|} \gtdim R\Hom(\s_{\mu^{*}}(\CC^N/\V), \s_{\lambda-s}\V)f1_{(k,N-k)}([\s_{\mu}\V])). \label{eq 51}
	\end{split}
\end{align}

From Corollary \ref{Corollary 7}, we have $f1_{(k,N-k)}([\s_{\mu}\V])$ is non-zero only when $\mu_{i} \neq i$ for all $i$, and for such $\mu$ there exists an unique $0 \leq t(\mu) \leq k$ such that $f1_{(k,N-k)}([\s_{\mu}\V])=(-1)^{t}[\s_{(\mu_{1}-1,...,\mu_{t(\mu)}-1,t(\mu),\mu_{t(\mu)+1},...,\mu_{k})}\V'']$. For simplicity, we denote $(\mu_{1}-1,...,\mu_{t(\mu)}-1,t(\mu),\mu_{t(\mu)+1},...,\mu_{k})$ to be $\mu(t)$. We have the following observations which will be used in the later analysis of the matrix coefficient.
\begin{enumerate}
	\item If $t(\mu)=0$, then we have $\mu_{i}=0$ for all $i$ and thus $\mu^{*}_{i}=0$ for all $i$.
	\item If $t(\mu) \geq 1$, then we have $\mu_{1} \geq 2$ (since $\mu_{1} \neq 1$)  and thus $\mu^{*}_{1} \geq 1$.
\end{enumerate}

Thus (\ref{eq 51}) becomes 
\begin{align*}
&g^{s}(\sum_{\substack{\mu \in P(N-k,k) \\ \mu_{i} \neq i \ \text{for all} \ i} } (-1)^{|\mu|+t(\mu)} \gtdim R\Hom(\s_{\mu^{*}}(\CC^N/\V), \s_{\lambda-s}\V)[\s_{\mu(t)}\V''] ) \\
& = \sum_{\substack{\mu \in P(N-k,k) \\ \mu_{i} \neq i \ \text{for all} \ i} } (-1)^{|\mu|+t(\mu)} \gtdim R\Hom(\s_{\mu^{*}}(\CC^N/\V), \s_{\lambda-s}\V)[\s_{\mu(t)+s}\V'']  \\
& = \sum_{\substack{\mu \in P(N-k,k) \\ \mu_{i} \neq i \ \text{for all} \ i \\ \lambda'' \in P(N-k-1,k+1)} }  (-1)^{|\mu|+|\lambda''|+t(\mu)} \gtdim R\Hom(\s_{\mu^{*}}(\CC^N/\V), \s_{\lambda-s}\V)\gtdim R\Hom(\s_{\lambda''^{*}}(\CC^N/\V''), \s_{\mu(t)+s}\V'')[\s_{\lambda''}\V''] 
\end{align*}

So the matrix coefficient $[f_{s}1_{(k,N-k)}]_{[\lambda],[\lambda'']}$ is given by 
\begin{align} \label{eq 52}
\begin{split}
&\sum_{\substack{\mu \in P(N-k,k) \\ \mu_{i} \neq i \ \text{for all} \ i }}  (-1)^{|\mu|+|\lambda''|+t(\mu)} \gtdim R\Hom(\s_{\mu^{*}}(\CC^N/\V), \s_{\lambda-s}\V)\gtdim R\Hom(\s_{\lambda''^{*}}(\CC^N/\V''), \s_{\mu(t)+s}\V'') \\
&=\sum_{\substack{\mu \in P(N-k,k) \\ \mu_{i} \neq i \ \text{for all} \ i }}  (-1)^{|\mu|+|\lambda''|+t(\mu)} \gtdim H^*(Fl, \LL_{x'}) \gtdim H^*(Fl,\LL_{y'}).
\end{split}
\end{align} by the same argument as in Theorem \ref{Theorem 10}, where $x'_{\mu}=(\lambda_{k}-s,...,\lambda_{1}-s,-\mu^{*}_{1},...,-\mu^{*}_{N-k})$ and  $y'_{\mu}=(\mu_{k}+s,...,\mu_{t(\mu)+1}+s,t(\mu)+s,\mu_{t(\mu)}-1+s,...,\mu_{1}-1+s,-\lambda_{1}^{''*},...,-\lambda^{''*}_{N-k-1})$.

The sum in (\ref{eq 52}) is non-zero only when $H^*(Fl, \LL_{x'_{\mu}}) \neq 0$ and $H^*(Fl,\LL_{y'_{\mu}}) \neq 0$, and we can use the Borel-Weil-Bott theorem (Theorem \ref{Theorem 4}) to calculate it just like we did in Theorem \ref{Theorem 10} for calculating $[e_{r}1_{(k,N-k)}]_{[\lambda],[\lambda']}$. However, the situation is a bit different here since there is a new parameter $t(\mu)$. We have to study the above two observations (1) and (2) about $t(\mu)$.  

For (1), we have $t(\mu)=0$ and so $\mu_{i}=\mu^{*}_{j}=0$ for all $i,j$. In this case we have $x'_{0}=(\lambda_{k}-s,...,\lambda_{1}-s,0,...,0)$ and  $y'_{0}=(s,...,s,-\lambda_{1}^{''*},...,-\lambda^{''*}_{N-k-1})$. Using Theorem \ref{Theorem 4}, we have the following result.

\begin{align}
	\begin{split}
	&H^*(Fl, \LL_{x'_{0}}) \otimes  H^*(Fl,\LL_{y'_{0}}) \\
	&= \begin{cases}
	(\s_{-x'_{0}}\CC^N)^{\vee} \otimes(\s_{-y'_{0}}\CC^N)^{\vee} & \text{if} \ s \geq \lambda_{1} \ \text{and} \ s \leq \lambda^{''*}_{1}  \\
	(\s_{-x'_{0}}\CC^N)^{\vee} \otimes(\s_{w \star (-y'_{0})}\CC^N)^{\vee}[-l(w)] &  \parbox[t]{.3\textwidth}{ if $s \geq \lambda_{1}$  and  $s > \lambda^{''*}_{1}$ with $\exists !  \ w \in W$ s.t. $w \star (-y'_{0})$ dominant} \\
	(\s_{w' \star (-x'_{0})}\CC^N)^{\vee} \otimes(\s_{-y'_{0}}\CC^N)^{\vee}[-l(w')] &  \parbox[t]{.3\textwidth}{ if $s < \lambda_{1}$  and  $s \leq \lambda^{''*}_{1}$ with $\exists !  \ w' \in W$ s.t. $w' \star (-x'_{0})$ dominant} \\
	(\s_{w_{1} \star (-x'_{0})}\CC^N)^{\vee} \otimes(\s_{w_{2} \star (-y'_{0})}\CC^N)^{\vee}[-l(w_{1})-l(w_{2})] &  \parbox[t]{.3\textwidth}{ if $s < \lambda_{1}$  and  $s >\lambda^{''*}_{1}$ with $\exists !  \ w_{1}, w_{2} \in W$ s.t. $w_{1} \star (-x'_{0})$, $w_{2} \star (-y'_{0})$ dominant} \\
	0  & \ \text{otherwise}
	\end{cases} 
	\end{split} \label{eq 56}
\end{align}

For (2), we have $t(\mu) \geq 1$ and $\mu_{1} \geq 2$. The analysis of this case is pretty much the same as the analysis in Theorem \ref{Theorem 10}. So we leave the details to the reader. The result is the following.

\begin{align} \label{eq 57}
	\begin{split}
	&\sum_{\mu \in P'_{1}(N-k,k)}  (-1)^{|\mu|+|\lambda'|+t(\mu) +l(w)} \dim (\s_{-x'_{\mu}}\CC^N)^{\vee} \dim  (\s_{w \star (-y'_{\mu})}\CC^N )^{\vee} \\
	&+\sum_{ \mu \in P'_{2}(N-k,k)}  (-1)^{|\mu|+|\lambda'|+t(\mu) +l(w')} \dim (\s_{w' \star (-x'_{\mu})}\CC^N)^{\vee} \dim  (\s_{-y'_{\mu}}\CC^N )^{\vee} \\
	&+\sum_{ \mu \in P'_{3}(N-k,k)}  (-1)^{|\mu|+|\lambda'|+t(\mu) +l(w_{1})+l(w_{2})} \dim (\s_{w_{1} \star (-x'_{\mu})}\CC^N)^{\vee} \dim  (\s_{w_{2} \star (-y'_{\mu})}\CC^N )^{\vee} 
	\end{split}
\end{align}

Combining (\ref{eq 56}) and (\ref{eq 57}), we prove the theorem.

\end{proof}

\begin{remark}
	Similarly like Theorem \ref{Theorem 10} the subsets $P'_{i}(N-k,k)$ can be empty depends on the value of $s, \lambda_{1}, \lambda^{''*}_{1}$ where $i=1,2,3$, in such case the sum would just be zero.
\end{remark}

\end{document}